\documentclass[english]{article}
\usepackage[T1]{fontenc}
\usepackage[latin9]{inputenc}
\usepackage{geometry}
\usepackage{float}
\usepackage{units}
\usepackage{mathtools}
\usepackage{amsmath}
\usepackage{amsthm}
\usepackage{amssymb}
\usepackage[authoryear]{natbib}

\makeatletter
\floatstyle{ruled}
\newfloat{algorithm}{tbp}{loa}
\providecommand{\algorithmname}{Algorithm}
\floatname{algorithm}{\protect\algorithmname}

\theoremstyle{plain}
\newtheorem{thm}{\protect\theoremname}
\theoremstyle{plain}
\newtheorem{lem}[thm]{\protect\lemmaname}
\ifx\proof\undefined
\newenvironment{proof}[1][\protect\proofname]{\par
	\normalfont\topsep6\p@\@plus6\p@\relax
	\trivlist
	\itemindent\parindent
	\item[\hskip\labelsep\scshape #1]\ignorespaces
}{%
	\endtrivlist\@endpefalse
}
\providecommand{\proofname}{Proof}
\fi

\usepackage{hyperref}

\makeatother

\usepackage{babel}
\providecommand{\lemmaname}{Lemma}
\providecommand{\theoremname}{Theorem}

\begin{document}
\global\long\def\E{\mathbb{E}}%
\global\long\def\F{\mathcal{F}}%
\global\long\def\R{\mathbb{R}}%
\global\long\def\tn{\widetilde{\nabla}f}%
\global\long\def\hn{\widehat{\nabla}f}%
\global\long\def\n{\nabla f}%
\global\long\def\indicator{\mathbf{1}}%
\global\long\def\mf{f(x^{*})}%
\global\long\def\breg{\mathbf{D}_{\psi}}%
\global\long\def\dom{\mathcal{X}}%
\global\long\def\norm#1{\left\lVert #1\right\rVert }%

\title{High Probability Convergence of Clipped-SGD Under Heavy-tailed Noise }

\author{ 
Ta Duy Nguyen\thanks{Equal contribution. Department of Computer Science, Boston University,
\texttt{{taduy@bu.edu}.}}  \\ \and 
Thien Hang Nguyen\thanks{Equal contribution. Khoury College of Computer and Information Science, Northeastern University,
\texttt{{nguyen.thien@northeastern.edu}.}}  \\ \and
Alina Ene\thanks{Department of Computer Science, Boston University, \texttt{{aene@bu.edu}.}} \\ \and 
Huy L. Nguyen\thanks{Khoury College of Computer and Information Science, Northeastern University,
\texttt{{hu.nguyen@northeastern.edu}.}}}

\date{}
\maketitle
\allowdisplaybreaks
\begin{abstract}
While the convergence behaviors of stochastic gradient methods are
well understood \emph{in expectation}, there still exist many gaps
in the understanding of their convergence with \emph{high probability},
where the convergence rate has a logarithmic dependency on the desired
success probability parameter. In the \emph{heavy-tailed
	noise} setting, where the stochastic gradient noise only has bounded
$p$-th moments for some $p\in(1,2]$, existing works could only show
bounds \emph{in expectation} for a variant of stochastic gradient
descent (SGD) with clipped gradients, or high probability bounds in
special cases (such as $p=2$) or with extra assumptions (such as
the stochastic gradients having bounded non-central moments). In this
work, using a novel analysis framework, we present new and time-optimal
(up to logarithmic factors) \emph{high probability} convergence bounds
for SGD with clipping under heavy-tailed noise for both convex and
non-convex smooth objectives using only minimal assumptions. 
\end{abstract}

\section{Introduction}

Stochastic gradient descent (SGD) is at the heart of many stochastic
optimization algorithms in modern machine learning. Studying the properties
of SGD and the conditions for its convergence is therefore of
great interest. Many classical works \citep{ghadimi2013stochastic,nemirovski2009robust}
study SGD under the assumption that the stochastic gradient noise
follows a light-tailed distribution (e.g. sub-Gaussian) or has bounded
variance. A recent line of works on deep learning problems \citep{zhang2020adaptive,simsekli2019tail}
suggests that this assumption may not hold in practice. Instead, they
show that for a variety of modern learning tasks, such as training
attention models like BERT \citep{zhang2020adaptive} and
convolutional networks \citep{simsekli2019tail}, the gradient noise
distribution behaves closer to that of a \emph{heavy-tailed distribution,}
where the variance and other moments of the noise can be extremely
large or even unbounded. This mismatch between theory and practice
manifests itself through the sub-optimal performances of SGD in certain
settings such as the long-standing failure case of SGD in recurrent
neural networks \citep{pascanu2012understanding} or the dominance
of adaptive methods like Adam over SGD in more modern settings \citep{zhang2020adaptive}.

These challenges to SGD brought about by heavy-tailed noise do not stop at the practical level. In the light tails setting, well-behaving moments (including the variance) allows the use of a variety of concentration techniques to control the noisy iterates of SGD. However, with \emph{ heavy-tailed noise}, the convergence of SGD and its variants proves challenging to analyze as the nice behaviors of light-tailed noise models no longer hold, often requiring additional assumptions and/or modifications to the algorithm. These challenges are evident by the limited existing works in this setting.

In our work, we tackle the theoretical question of establishing high
probability convergence of stochastic first-order methods in the heavy-tailed
noise regime. More specifically, we analyze the heavy-tailed noise
model proposed by \citet{zhang2020adaptive} in which the gradient
noise has \emph{bounded $p$-th moment}:
\[
\E[\|\widehat{\nabla}f(x)-\nabla f(x)\|^{p}]\leq\sigma^{p},\]
for some $\sigma>0$, and $p\in(1,2].$ Here, $\widehat{\nabla}f(x)$
is a stochastic estimate to the true gradient $\nabla f(x)$ of the
function $f$ of interest at point $x$. When $p=2$, this recovers
the bounded variance setting, a common and well-studied assumption
for the analyses in expectation of stochastic gradient methods. However,
when $p<2$, \citet{zhang2020adaptive} show that there are scenarios
where SGD \emph{fails to converge} to a stationary point even in expectation.
There, the presence of large stochastic gradients is the main culprit
for the non-convergence of SGD. The authors then show that SGD with
appropriate \emph{gradient clipping }(or \emph{clipped-SGD}) not only
alleviates this problem but also attains the \emph{optimal} convergence
rate in expectation for non-convex smooth objectives, matching their
lower bound. 

However, convergence in expectation is often unsatisfactory due to
its implication that the convergence is only guaranteed when one can
perform multiple runs of the algorithm, whereas typically in machine
learning problems performing multiple runs can incur significant computational
and statistical costs \citep{harvey2019tight,madden2020high,davis2021low}.
Hence, a high probability convergence guarantee, where the convergence
rate has a logarithmic dependency on the success probability, is highly
desirable.

In the heavy-tailed noise regime, where the stochastic gradient noise
has bounded $p$-th moment for some $p\in(1,2]$, results on high
probability convergence of SGD variants are limited. To the best of
our knowledge, \citet{cutkosky2021high} is the only work that provides
high probability bounds for the convergence of clipped-SGD \emph{with
momentum} in the non-convex setting, but relying on the additional
assumption that the stochastic gradients are well behaved. More precisely,
this work assumes that the stochastic gradients have uniformly bounded
\emph{non-central} $p$-th moment: $\E[\|\widehat{\nabla}f(x)^{p}\|]\leq\sigma^{p}$. However, this is a strong assumption that implies the true gradients
are bounded and excludes important objective functions such as quadratic
functions. This work has another drawback of not achieving noise-adaptive
rates that improve towards the deterministic rate as the amount of
noise decreases. Other recent works by \citet{gorbunov2020stochastic}
and \citet{nazin2019algorithms} show high probability convergence
rates for smooth convex optimization in the special case of bounded
variance ($p=2$), with the latter having to assume that the optimization
domain has bounded diameter. Thus, establishing high probability convergence
guarantees under minimal assumptions remains open in both the convex
and non-convex regimes.

\subsection{Contributions}

Our work fills in the aforementioned gaps in the study of high probability
convergence of stochastic gradient methods with heavy-tailed stochastic
gradient noise with bounded $p$-th moment for
both convex and non-convex optimization.

$\bullet$ We show that a simple clipping strategy with appropriate choices
of step sizes and clipping parameters is sufficient to ensure   convergence with high probability
 for both convex and non-convex optimization.
Our clipping strategy does not employ  momentum, yet still achieves
optimal dependency on the time horizon without any assumptions on
the gradients or stochastic gradients.

$\bullet$ In the convex setting, we provide the first high probability convergence
rate for clipped-SGD under noises with bounded $p$-th moment for
 $p\in(1,2]$. Our result generalizes the result from \citet{gorbunov2020stochastic},
where high probability convergence rate is only shown for $p=2$.
Our convergence guarantee is $O(T^{\frac{1-p}{p}})$, which is time-optimal
(up to logarithmic factors) and matches the lower bounds proven in
\citet{raginsky2009information,vural2022mirror}. 

$\bullet$ In the non-convex setting, we provide a high probability bound for
the convergence of clipped-SGD to a stationary point. Our convergence
guarantee is $O(T^{\frac{2-2p}{3p-2}})$, which is also time-optimal
(up to logarithmic factors) and matches the lower bound in \citet{zhang2020adaptive}.
This result complements the convergence in expectation of clipped-SGD
provided by \citet{zhang2020adaptive}. In contrast to the prior work
of \citet{cutkosky2021high} that strongly relies on the assumption
that the stochastic gradients have bounded non-central moments, our
analysis does not make any assumptions on the gradients or the stochastic
gradients.

$\bullet$ Our work builds upon the line of work on high probability bounds for
stochastic gradient methods that we discuss in more detail in Section
\ref{subsec:related-works}. Our approach and techniques apply to
the challenging setting of heavy-tailed noise with bounded $p$-th moment,
which extends beyond the light-tailed and bounded variance noise settings
 considered in prior works. Our approach is general and can be applied
to both convex and non-convex optimization, with minimal assumptions.

\subsection{Our techniques }

Our work builds upon the line of works that analyze the high-probability
convergence of SGD methods that we discuss in more detail in Section
\ref{subsec:related-works}. Here, we highlight some of the main challenges in
the heavy-tailed setting and describe our key techniques
for overcoming them.

Our work departs from the most closely related works \citep{gorbunov2020stochastic,cutkosky2021high}
in several key aspects.  While our approach is built upon \citep{gorbunov2020stochastic}, this work only applies
to the special case of bounded variance ($p=2$) in the convex setting.
For the non-convex setting, \citet{cutkosky2021high} rely
on the use of momentum in their analysis to obtain the right convergence along with additional assumptions on the stochastic gradients. 

We highlight three challenges that arise when trying to extend the
aforementioned works. For concreteness, let us denote by $\tn(x)$ the clipped gradient,
which is a biased estimate of the true gradient $\nabla f(x)$ of
the function $f$ of interest at point $x$. As noted earlier, clipping
is essential for mitigating the effects of heavy-tailed noise, but
it introduces bias in the gradient estimate. The first challenge is
analyzing both the bias $\|\E[\tn(x)]-\nabla f(x)\|$
and the variance $\E[\| \tn(x)-\E[\tn(x)]\| ^{2}]$
of the clipped gradient estimation without relying on the bounded
variance assumption used in the prior work \citep{gorbunov2020stochastic}.
Here, new techniques are needed to handle the general setting
of bounded $p$-th moment. The second challenge arises in obtaining the
optimal convergence rate in the non-convex setting without relying
on the addition of momentum for variance reduction as well as strong
assumptions on the gradients and the stochastic gradients. Simply
applying the techniques developed in the convex setting will not give
the optimal convergence rate, so new insights are needed. The third
challenge arises in the derivation of the clipping parameters that
lead to the optimal convergence rates, where our analysis provides
a key insight: the appropriate choices highly depend on $p$,
 which differ significantly between the convex and non-convex setting.
We now discuss some of the key techniques to overcome
these challenges.

\noindent \textbf{Inductive argument:} At the center of our approach is an induction
on the number of iterations $N$ showing that, with high probability,
all of the stochastic quantities of interest that arise in the first
$N$ iterations are well-controlled. In the convex setting, similarly
to \citep{gorbunov2020stochastic}, a key component is the analysis
of the distances between the iterates and the optimum, and we show
that these distances are all bounded by a constant with high probability.
In the non-convex setting, the induction differs significantly: instead
of the iterates' distances to the optimum, we show a constant bound
for the function value gaps.

Our inductive argument is a significant departure from the prior work
\citep{cutkosky2021high} for non-convex optimization. Notably, the
strong assumption that the non-central $p$-th moments of the stochastic
gradients are bounded that is employed in \citet{cutkosky2021high}
implies that the true gradients are also bounded. Due to these assumptions,
Freedman's inequality can be directly applied without significant concerns,
leading to an arguably much simpler analysis. Without these strong
assumptions, there is no longer any a priori bound on the gradients,
and the conditions required by Freedman's inequality are not a priori
satisfied. Instead, we show that the true gradients are bounded as
part of our inductive argument. Our techniques for bounding the relevant
stochastic quantities are general and could be applied to the analysis
of a broader class of algorithms, including those with momentum, which
we leave as future work.

\noindent\textbf{Analysis of the bias-variance trade-off of clipped gradients:}
A key tool in our analysis is Lemma \ref{lem:Gorbunov-F5-p}, where
we derive appropriate bounds for the bias and variance of the clipped
stochastic gradient $\tn(x)$.  The lemma quantifies how the clipping
parameter controls the trade-off between the bias of
the clipped gradient and its variance: a smaller clipping parameter results in a smaller variance but also 
a higher bias. Contrast this to vanilla
SGD under heavy-tailed noise: while its estimate is unbiased, it suffers
from too high a variance, leading to non-convergence under certain
conditions \citep{zhang2020adaptive}. The ability for clipped gradients
to trade variance for bias allows for clipped-SGD to converge in this
regime. Balancing this trade-off appropriately is central to our analysis.

Also, note that these bounds hold only under the condition that the
true gradient norm $\| \nabla f(x)\| $ is not too
large. Here, the induction helps us overcome this challenge: by conditioning
on the inductive hypothesis, the bounds for the bias and variance
of $\tn(x)$ hold. Lemma \ref{lem:Gorbunov-F5-p} is also important
for arguing about the probability of the event that the conditions
for applying Freedman's inequality are satisfied. An insight from
this analysis is that most often we do not need to clip the gradients,
and thus the clipping algorithm exhibits the same behaviors as SGD
with high probability.

\noindent\textbf{Application of Freedman's inequality:} A primary tool for
establishing bounds on the iterates is Freedman's inequality for the
sum of bounded martingale difference sequences:
\begin{lem}[Freedman's inequality]
\label{thm:freedman}Let $(X_{t})_{t\ge1}$ be a martingale difference
sequence. Assume that there exists a constant $c$ such that $\left|X_{t}\right|\le c$
almost surely for all $t\ge1$ and define $\sigma_{t}^{2}=\E\left[X_{t}^{2}\mid X_{t-1},\dots,X_{1}\right]$.
Then for all $b>0$, $F>0$ and $T\ge1$ 
\begin{align*}
\Pr\left[\left|\sum_{t=1}^{T}X_{t}\right|>b\text{ and }\sum_{t=1}^{T}\sigma_{t}^{2}\le F\right] & \le2\exp\left(-\frac{b^{2}}{2F+2cb/3}\right).
\end{align*}
\end{lem}
We use Freedman's inequality to analyze the deviation term $\tn(x)-\E[\tn(x)]$
of the clipped stochastic gradients appearing in our analysis.  In order to apply Freedman's
inequality, the elements of the martingale difference sequence have
to be bounded. Unfortunately, this necessary condition fails to hold
for one of the sequences obtained from the standard analysis of
SGD without further assumptions. In \citep{cutkosky2021high}, the
necessary conditions for the application of Freedman's inequality
are established via the strong assumption that the non-central $p$-th
moments of the stochastic gradients are bounded, allowing for a direct
application of Freedman's inequality. On the other hand, we apply
Freedman's on a surrogate sequence that is bounded. Then, utilizing
our inductive hypothesis and Lemma \ref{lem:Gorbunov-F5-p}, we argue
that our original sequence is the same as that surrogate sequence
(and hence also controlled by Freedman's) with high probability.

\noindent\textbf{Separation between the convex and non-convex setting:} Since
the lower bounds for the convex \citep{vural2022mirror} and non-convex
\citep{zhang2020adaptive} settings are different, there must exist
a fundamental difference between minimizing a convex function and
finding a stationary point of a non-convex function under heavy-tailed noise. 
Our analysis
provides insights into this difference at a technical level. In the convex setting, the goal is to upper bound the function value gaps, and the main stochastic error involves the dot product of the distance from the iterates to the optimum and the gradient estimate error. In contrast, in the non-convex setting, the goal is to upper bound the squared norms of the gradients, which is a weaker requirement than bounding the function value gaps. In this latter setting, the main stochastic error involves a different quantity: the dot product of the true gradient and the gradient estimate error. A natural approach is to apply the same strategy from the convex setting to the non-convex setting. However, this
leads to a sub-optimal convergence rate for the non-convex setting
that does not match the lower bound. The key to improving the bound is to realize that,
in the non-convex setting, we can use part of the norm squared of
the gradients to absorb the dot product containing the estimation error. We discuss this in more details in Section \ref{sec:Non-convex}.

We also highlight another key difference between our work and that
of \citet{cutkosky2021high} for non-convex objectives. The latter
heavily relies on the use of momentum to obtain the optimal convergence
rate. In contrast, by leveraging the gain and loss in the function
analysis, as described above, we achieve the right convergence
rate for the simpler SGD variant, clipped-SGD.

\subsection{Related works}

\label{subsec:related-works}

\textbf{Heavy-tailed noise, gradient clipping, and high-probability
convergence:} Clipped-SGD has been long utilized to ameliorate gradient
explosion problems \citep{pascanu2012understanding}. Recently, a
motivation for the development of clipped-SGD methods is that, as
pointed out by \citet{zhang2020adaptive}, vanilla SGD does not converge,
even in expectation, when the gradient noise has infinite variance.
Further, \citet{csimcsekli2019heavy,simsekli2019tail,zhang2020adaptive,gurbuzbalaban2021heavy}
suggest that in practical deep learning problems, gradient noises
exhibit heavy tails. In particular, \citet{zhang2020adaptive}
provide empirical evidence for that in attention models and prove the convergence
in expectation of clipped-SGD for strongly convex and nonconvex smooth
objectives. 

The convergence in expectation of vanilla SGD has been studied
for noises with bounded variance in the works \citep{ghadimi2013stochastic,nemirovski2009robust,khaled2020better}.
Using different gradient clipping strategies, \citet{nazin2019algorithms,gorbunov2020stochastic}
show that, in the convex setting, their variants of SGD converge
with high probability in the bounded variance regime ($p=2)$. The former
work requires a bounded domain assumption while the latter is for unconstrained optimization and requires knowledge of the initial
distance. More recently, for the problem of finding stationary points
of non-convex objectives, \citet{cutkosky2021high} demonstrate the
convergence with high probability of another variant of clipped-SGD
with momentum, under an additional assumption that the non-central
$p$-th moments of the stochastic gradients are also bounded -- a
relatively strong assumption that implies the true gradients are bounded,
which excludes objectives such as quadratic functions. This assumption
is key to employing Freedman's inequality in their analysis. Moreover,
\citet{cutkosky2021high} rely on momentum to
obtain the optimal convergence rate. By comparison, our analysis utilizes
an induction argument as well as the bias-variance trade-off to obtain
the optimal rate without momentum or additional assumptions.

The question of under which additional
condition(s)  vanilla SGD can converge with heavy-tailed noise has
also been studied. \citet{wang2021convergence} show that vanilla
SGD converges\emph{ }in expectation for a special type of strongly
convex functions under heavy-tailed noise. In a similar setting,
\citet{vural2022mirror} provide an in expectation\emph{ }convergence
rate for stochastic mirror descent for strongly convex problems with
compact domain under the stronger  bounded non-central $p$-moments stochastic
gradient assumption.

In our work, we analyze the convergence with \emph{high probability}
of clipped-SGD  under heavy-tailed noise. Our approach is applicable to both
convex and non-convex regimes using only the standard assumptions,
while achieving optimal rates (up to logarithmic factors)
in both cases (see \citet{raginsky2009information,vural2022mirror}
for lower bounds in the convex setting). Our techniques are developed
based on the work on convex objectives by \citet{gorbunov2020stochastic}
for noises with bounded variance, extending it to noises with bounded
$p$-th moment in both convex and non-convex regimes.

\textbf{High-probability convergence for light-tailed noise:} With
light-tailed noises, gradient clipping is not necessary for stochastic
gradient methods to achieve high probability bounds. Convergence in
high probability for stochastic mirror descent and stochastic gradient
descent with sub-Gaussian noise has been established in the works
by \citet{nemirovski2009robust,lan2012optimal,kakade2008generalization,rakhlin2011making,hazan2014beyond,harvey2019tight,dvurechensky2016stochastic}.
These works require that the domain has bounded diameter (or Bregman
diameter in the case of mirror descent) or assume strong convexity
of the objective. In the nonconvex setting, \citet{li2020high}, \citet{madden2020high}
and \citet{li2022high} show the convergence of various variants of
SGD for sub-Gaussian noises and generalize to the family of sub-Weibull
distributions, which have lighter tails than the case we consider
in this work.

\section{Preliminaries \label{sec:Preliminaries}}
We consider the unconstrained problem of minimizing a differentiable
function $f:\R^{d}\to\R$ over $\R^{d}$. Throughout the paper, we
make the following standard assumptions:\\
\textbf{(1) Unbiased estimator}: Instead of having direct access to
$f$, we assume that our algorithm is allowed to query $f$ via
a stochastic first-order oracle that returns a history-independent,
unbiased gradient estimator $\hn(x)$ of $\nabla f(x)$ for any
$x\in\R^{d}$. That is, conditioned on the history and the queried
point $x$, we have $\E[\hn(x)\mid x]=\n(x)$.\\
\textbf{(2) Bounded $p$-th moment noise}: We assume that for any
$x\in\R^{d}$, $\hn(x)$ satisfies  
$
\E[\| \hn(x)-\n(x)\| ^{p}\mid x] \le\sigma^{p},$
 for some $\sigma>0$ and some $p \in (1,2]$. This is commonly referred
to as \emph{heavy-tailed} noise, as opposed to light-tailed noise
such as those that are distributed according to sub-Gaussian or sub-Exponential
distributions.\\
\textbf{(3) $L$-smoothness}: We consider the class of $L$-smooth
functions. We say that $f$ is $L$-smooth if its gradients are $L$-Lipschitz,
i.e. for all $x,y\in\R^{d}$, $
\left\Vert \nabla f(x)-\nabla f(y)\right\Vert \le L\left\Vert x-y\right\Vert.
$
Throughout the paper, $\left\Vert \cdot\right\Vert $ denotes the
$\ell_{2}$ norm. We will utilize the following quadratic upperbound
for $L$-smooth functions: for all $x,y\in\R^{d}$, $
f(y)\le f(x)+\left\langle \nabla f(x),y-x\right\rangle +\frac{L}{2}\left\Vert y-x\right\Vert ^{2}.$\\
\textbf{(4.1) Existence of a minimizer}: In the convex setting, we
assume that there exists a minimizer $x^*$ of $f$: $x^{*}\in\arg\min_{x\in\R^{d}}f(x)$. We let $f^{*}:=f(x^{*})$.\\
\textbf{(4.2) Finite lower bound for the function value}: In the non-convex
setting, we assume that there exists $f^{*}$ such that $-\infty<f^{*}\le\inf_{x\in\R^{d}}f(x)$. This assumption is used in both the lower bound and
upper bound in \citet{zhang2020adaptive}.

\section{Algorithm and noise bounds \label{sec:Algorithm}}

\begin{algorithm}
\caption{Clipped-SGD}
\label{alg:clipped-sgd}

Parameters: initial point $x_{1}$, step sizes $\left\{ \eta_{t}\right\} $,
clipping parameters $\left\{ \lambda_{t}\right\} $

for $t=1$ to $T$ do

$\quad$ $\tn(x_{t})=\min\left\{ 1,\frac{\lambda_{t}}{\left\Vert \hn(x_{t})\right\Vert }\right\} \hn(x_{t})$

$\quad$ $x_{t+1}=x_{t}-\eta_{t}\tn(x_{t})$
\end{algorithm}

SGD with gradient clipping \citep{zhang2020adaptive,gorbunov2020stochastic} is presented in Algorithm \ref{alg:clipped-sgd}.
In each iteration, the algorithm performs the standard SGD update using $\tn(x_{t})$
with step size $\eta_{t}$, where $\tn(x_{t})$ clips the stochastic
gradient $\hn(x_{t})$ according to the clipping parameter
$\lambda_{t}$:
\begin{equation}
\tn(x_{t}):=\min\left\{ 1,\frac{\lambda_{t}}{\left\Vert \hn(x_{t})\right\Vert }\right\} \hn(x_{t}).\label{eq:clipping-formula}
\end{equation}
The advantage of using $\tn(x_{t})$ instead of $\hn(x_{t})$ is that
$\tn(x_{t})$ has a bounded norm and a lower variance, allowing us to mitigate the
effects of large stochastic gradients on the convergence of the algorithm.
However, the clipped gradient is a biased estimate
of the true gradient $\n(x)$. Therefore, we will have to handle the
error caused by $\tn(x_{t})$ more carefully. In Lemma \ref{lem:Gorbunov-F5-p},
the clipping parameter $\lambda_t$ quantifies this bias-variance trade-off for $\tn(x_{t})$, and an appropriate balancing will be central
to our analysis of the clipped-SGD algorithm. We note that Lemma \ref{lem:Gorbunov-F5-p}
holds regardless of the convexity or smoothness of the objective $f$.

For simplicity, in the rest of the paper, we define the following
notations. For $t\ge1$, we define the function value gap $\Delta_{t}:=f(x_{t})-f^{*}$.
In the convex setting, we let $R_{t}:=\left\Vert x_{t}-x^{*}\right\Vert $
be the distance from the point $x_{t}$ and the optimum $x^{*}$.
We will also use $\E_{t}\left[\cdot\right]$ to denote the expectation
conditioned on all the randomness up to (but not including) iteration
$t$. Finally, let
\begin{align}
 & \theta_{t}:=\tn(x_{t})-\nabla f(x_{t});\quad\theta_{t}^{u}:=\tn(x_{t})-\E_{t}\left[\tn(x_{t})\right];\quad\theta_{t}^{b}:=\E_{t}\left[\tn(x_{t})\right]-\nabla f(x_{t}).\label{eq:notations-for-grad-est}
\end{align}
Note that $\theta_{t}^{u}+\theta_{t}^{b}=\theta_{t}.$ With these
notations, we present Lemma \ref{lem:Gorbunov-F5-p} below (proof in Section \ref{sec:Proof-for-sec-algorithm}). This lemma extends Lemma F5 from
\citep{gorbunov2020stochastic} and Lemma 10 from \citep{zhang2020adaptive}.
\begin{lem}
\label{lem:Gorbunov-F5-p}For $t\ge1$, for $\hn(x_{t})$ satisfying
assumption (2) and $\tn(x_{t})$ defined in (\ref{eq:clipping-formula}),
we have

\begin{align}
\left\Vert \theta_{t}^{u}\right\Vert =\left\Vert \tn(x_{t})-\E_{t}\left[\tn(x_{t})\right]\right\Vert  & \le2\lambda_{t}\label{eq:1-p}
\end{align}
Furthermore, if $\left\Vert \nabla f(x_{t})\right\Vert \le\frac{\lambda_{t}}{2}$
then
\begin{align}
\left\Vert \theta_{t}^{b}\right\Vert = & \left\Vert \E_{t}\left[\tn(x_{t})\right]-\nabla f(x_{t})\right\Vert \le4\sigma^{p}\lambda_{t}^{1-p};\label{eq:2-p}\\
\E_{t}\left[\left\Vert \theta_{t}^{u}\right\Vert ^{2}\right]= & \E_{t}\left[\left\Vert \tn(x_{t})-\E_{t}\left[\tn(x_{t})\right]\right\Vert ^{2}\right]\le16\sigma^{p}\lambda_{t}^{2-p}.\label{eq:3-p}
\end{align}
\end{lem}

\section{Convergence of convex objectives \label{sec:Convex}}

In this section, we establish a high-probability convergence guarantee
for Algorithm \ref{alg:clipped-sgd} for smooth convex objectives
under heavy-tailed noise. Our analysis builds on the work \citet{gorbunov2020stochastic}
for the special case of noise with bounded variance ($p=2$), and
it extends this result to noise with bounded $p$-th moments for any $p\in(1,2]$.

\subsection{Main result}

Theorem \ref{thm:clipped-sgd-convergence-p-convex} provides a high
probability convergence guarantee of clipped-SGD (Algorithm \ref{alg:clipped-sgd})
in the convex setting in along with the choice of parameters.
\begin{thm}
\label{thm:clipped-sgd-convergence-p-convex}Assume $f$ is a convex
and differentiable function which satisfies assumptions (1), (2),
(3), and (4.1). With the choice 
\begin{align*}
\lambda_{t} & =\lambda=\max\left\{ (16T)^{1/p}\sigma;\sqrt{2}LR_{1}\right\} \text{, and }\\
\eta_{t} & =\eta=\frac{R_{1}}{16\lambda\ln\frac{4T}{\delta}}=\frac{R_{1}}{16\ln\frac{4T}{\delta}}\min\left\{ (16T)^{-1/p}\sigma^{-1};(\sqrt{2}LR_{1})^{-1}\right\} ,
\end{align*}
the iterate sequence $(x_{t})_{t\ge1}$ output by Algorithm \ref{alg:clipped-sgd}
satisfies
\begin{align*}
\frac{1}{T}\sum_{t=1}^{T}\left[f(x_{t})-f^{*}\right] & \le32R_{1}\ln\frac{4T}{\delta}\max\left\{ 16^{1/p}T^{\frac{1-p}{p}}\sigma;\sqrt{2}LR_{1}T^{-1}\right\} .
\end{align*}
\end{thm}
Note that, in the deterministic setting with $\sigma=0$, the step
size is $O\left(1/L\right)$ and the convergence rate is $O\left(1/T\right)$,
analogously to the step size and convergence rate of (deterministic)
gradient descent for smooth convex functions. Thus the above convergence
rate is adaptive to noise. In the stochastic setting with $\sigma>0$,
the convergence rate is $O\left(T^{\frac{1-p}{p}}\right)$. This is
the optimal bound for gradient methods with heavy tailed noise, as
proved in \citet{raginsky2009information,vural2022mirror}.

Similarly to \citet{gorbunov2020stochastic}, the theoretical choice
for the clipping parameter requires knowledge of a suitable upper
bound on problem parameters such as the initial distance $R_{1}=\left\Vert x_{1}-x^{*}\right\Vert $.
In contrast, the clipping strategy for non-convex objectives that
we study in the next section is very different, and it requires much
weaker information. We discuss these differences in more detail in
the following section.

\subsection{Proof overview}

Our work generalizes the techniques from \citet{gorbunov2020stochastic}
under the bounded variance assumption to the bounded $p$-moment noise
assumption for any $p\in(1,2]$. The analysis starts with the standard
function value analysis for convex smooth functions. We obtain the
following upper bound on the function value gaps, $\Delta_{t}:=f(x_{t})-f^{*}$.
We defer the proof to Section \ref{sec:appendix-convex-pfs} of the
Appendix.
\begin{lem}
\label{lem:convex-basic-analysis}Assuming that $\eta_{t}\le\frac{1}{4L}$
then for all $t\ge1$,
\begin{align}
\eta_{t}\Delta_{t} & \le\left\Vert x_{t}-x^{*}\right\Vert ^{2}-\left\Vert x_{t+1}-x^{*}\right\Vert ^{2}+2\eta_{t}^{2}\left\Vert \theta_{t}\right\Vert ^{2}-2\eta_{t}\left\langle \theta_{t},x_{t}-x^{*}\right\rangle .\label{eq:clipped-sgd-basic-inequality-convex}
\end{align}
Thus, by summing up over $t$, for every $k\ge1$
\begin{align}
\sum_{t=1}^{k}\eta_{t}\Delta_{t} & \le\left\Vert x_{1}-x^{*}\right\Vert ^{2}-\left\Vert x_{k+1}-x^{*}\right\Vert ^{2}+2\sum_{t=1}^{k}\eta_{t}^{2}\left\Vert \theta_{t}\right\Vert ^{2}-2\sum_{t=1}^{k}\eta_{t}\left\langle \theta_{t},x_{t}-x^{*}\right\rangle .\label{eq:clipped-sgd-convex-sum}
\end{align}
\end{lem}
We highlight the appearance of the term $\left\langle \theta_{t},x_{t}-x^{*}\right\rangle =\left\langle \tn(x_{t})-\nabla f(x_{t}),x_{t}-x^{*}\right\rangle $:
it is the primary reason for the slower rate achieved in the convex
setting compared to the non-convex setting. In the next section, we
will return to this point as we outline our analysis for the non-convex
case. Starting from equation (\ref{eq:clipped-sgd-convex-sum}), it
suffices to bound the sum $2\sum_{t=1}^{k}\eta_{t}^{2}\left\Vert \theta_{t}\right\Vert ^{2}-2\sum_{t=1}^{k}\eta_{t}\left\langle \theta_{t},x_{t}-x^{*}\right\rangle $
on the RHS with high probability. This is where we look to employ
the tools from Lemma \ref{lem:Gorbunov-F5-p} given by our clipped
gradient estimates, as well as extract appropriate martingale difference
sequences and control them using Freedman's inequality. To accomplish
this, we start with the following decomposition: for $N\ge1$, we
have
\begin{align*}
  2\eta^{2}\sum_{t=1}^{N}\left\Vert \theta_{t}\right\Vert ^{2}-2\eta\sum_{t=1}^{N}\left\langle x_{t}-x^{*},\theta_{t}\right\rangle 
& \leq  4\eta^{2}\sum_{t=1}^{N}\E_{t}\left[\left\Vert \theta_{t}^{u}\right\Vert ^{2}\right]+4\eta^{2}\sum_{t=1}^{N}\left\Vert \theta_{t}^{b}\right\Vert ^{2}+2\eta\sum_{t=1}^{N}\left\langle x_{t}-x^{*},\theta_{t}^{b}\right\rangle \\
 & +2\eta\sum_{t=1}^{N}\left\langle x_{t}-x^{*},\theta_{t}^{u}\right\rangle +4\eta^{2}\sum_{t=1}^{N}\left(\left\Vert \theta_{t}^{u}\right\Vert ^{2}-\E_{t}\left[\left\Vert \theta_{t}^{u}\right\Vert ^{2}\right]\right).
\end{align*}
The main argument in this proof is an induction on $N$ that, with
high probability, the terms above combined are bounded by $R_{1}^{2}$
for all $t\le N$. The induction gives a bound on $\Delta_{t}$ which
in turn is an upper bound for the length $\left\Vert \nabla f(x_{t})\right\Vert $
of the true gradient. Conditioning on $\left\Vert \n(x_{t})\right\Vert \leq\lambda/2$
(which can be achieved using appropriate parameter choice), the first
three terms in the RHS can be bounded by Lemma \ref{lem:Gorbunov-F5-p}.
For the last two terms, since each term forms a martingale difference
sequence, we use Freedman's inequality. By selecting suitable parameters,
we can obtain the desired bound.

\section{Convergence of non-convex objectives \label{sec:Non-convex}}

In this section, we present the main convergence guarantee to a stationary
point for clipped-SGD when the objective is non-convex. We also highlight
key differences between the convex and the non-convex settings to
give some intuition behind why a better convergence rate can be achieved
in the non-convex setting.

\subsection{Main result}

The choice of parameters and convergence guarantee of Algorithm \ref{alg:clipped-sgd}
in the non-convex setting is given in the following theorem. 
\begin{thm}
\label{thm:clipped-sgd-convergence-p-nonconvex}Assume $f$ satisfies
Assumptions (1), (2), (3) and (4.2), with the choice 
\begin{align*}
\lambda_{t} & =\lambda=\max\left\{ \left(\frac{8\ln\frac{4T}{\delta}}{\sqrt{L\Delta_{1}}}\right)^{\frac{1}{p-1}}T^{\frac{1}{3p-2}}\sigma^{\frac{p}{p-1}};4\sqrt{L\Delta_{1}};32^{1/p}\sigma T^{\frac{1}{3p-2}}\right\} \\
\eta_{t} & =\eta=\frac{\sqrt{\Delta_{1}}T^{\frac{1-p}{3p-2}}}{8\lambda\sqrt{L}\ln\frac{4T}{\delta}}=\frac{\sqrt{\Delta_{1}}}{8\sqrt{L}\ln\frac{4T}{\delta}}\min\left\{ \left(\frac{8\ln\frac{4T}{\delta}}{\sqrt{L\Delta_{1}}}\right)^{\frac{-1}{p-1}}T^{\frac{-p}{3p-2}}\sigma^{\frac{-p}{p-1}};\frac{T^{\frac{1-p}{3p-2}}}{4\sqrt{L\Delta_{1}}};\frac{T^{\frac{-p}{3p-2}}}{32^{1/p}\sigma}\right\} 
\end{align*}
the iterate sequence $(x_{t})_{t\ge1}$ output by Algorithm \ref{alg:clipped-sgd}
that satisfies
\begin{align*}
\frac{1}{T}\sum_{t=1}^{T}\left\Vert \nabla f(x_{t})\right\Vert ^{2} & \le32\sqrt{\Delta_{1}L}\ln\frac{4T}{\delta}\max\left\{ \left(\frac{8\ln\frac{4T}{\delta}}{\sqrt{L\Delta_{1}}}\right)^{\frac{1}{p-1}}T^{\frac{2-2p}{3p-2}}\sigma^{\frac{p}{p-1}};4\sqrt{L\Delta_{1}}T^{\frac{1-2p}{3p-2}};32^{1/p}\sigma T^{\frac{2-2p}{3p-2}}\right\} .
\end{align*}
\end{thm}
Recall from the previous section that, in the convex setting, the
parameters $\lambda$ and $\eta$ are set based on the initial distance
$R_{1}=\left\Vert x_{1}-x^{*}\right\Vert $. In contrast, for the non-convex
setting, we set the parameters based on the initial function value
gap $\Delta_{1}=f(x_{1})-f^{*}$. Recall that $f^{*}$ only needs
to be a lower bound for the function $f$, which is
readily computable  in many cases. For example, for over-parametrized models, the
loss is zero at the optimum, and we can set $f^{*}=0$. This means
the knowledge required to set the parameters is strictly less than in the convex setting. 

Note that the convergence guarantee
in Theorem \ref{thm:clipped-sgd-convergence-p-nonconvex} is ``almost''
adaptive to noise. When $\sigma=0$, it is $O\left(T^{\frac{1-2p}{3p-2}}\right)$
which is not the usual $O\left(T^{-1}\right)$ rate. However, we know
that if $\sigma=0$, any $p\in(1,2]$ will satisfy the noise
condition. By setting $p\to1$, we obtain the rate $O\left(T^{-1}\right)$. Finally, note that the convergence rate $O\left(T^{\frac{2-2p}{3p-2}}\right)$
is stronger than the optimal rate of $O\left(T^{\frac{1-p}{p}}\right)$
for the convex setting. This rate is also the optimal convergence
rate for non-convex functions under noises with bounded $p$-th moment
as proven in \citep{zhang2020adaptive}.

\subsection{Analysis of clipped-SGD in the non-convex setting }

The analysis starts with the following bound for the gradient $\left\Vert \nabla f(x_{t})\right\Vert ^{2}$.
\begin{lem}
\label{lem:nonconvex-basic-analysis}Assuming that $\eta_{t}\le\frac{1}{L}$
then for all $t\ge1$,
\begin{align}
\frac{\eta_{t}}{2}\left\Vert \nabla f(x_{t})\right\Vert ^{2} & \le\Delta_{t}-\Delta_{t+1}+L\eta_{t}^{2}\left\Vert \theta_{t}^{u}\right\Vert ^{2}+\left(L\eta_{t}^{2}-\eta_{t}\right)\left\langle \nabla f(x_{t}),\theta_{t}^{u}\right\rangle +\frac{3\eta_{t}}{2}\left\Vert \theta_{t}^{b}\right\Vert ^{2}\label{eq:clipped-sgd-basic-inequality-nonconvex}
\end{align}
Thus, by summing up over $t$, for every $k\ge1$
\begin{equation}
\sum_{t=1}^{k}\frac{\eta_{t}}{2}\left\Vert \nabla f(x_{t})\right\Vert ^{2}\le\Delta_{1}-\Delta_{T}+\sum_{t=1}^{k}L\eta_{t}^{2}\left\Vert \theta_{t}^{u}\right\Vert ^{2}+\sum_{t=1}^{k}\left(L\eta_{t}^{2}-\eta_{t}\right)\left\langle \nabla f(x_{t}),\theta_{t}^{u}\right\rangle +\sum_{t=1}^{k}\frac{3\eta_{t}}{2}\left\Vert \theta_{t}^{b}\right\Vert ^{2}.\label{eq:sum-nonconvex-basic}
\end{equation}
\end{lem}
This lemma will help us gain an insight into why there is a separation
between the convex and non-convex settings. The last two terms on
the RHS of (\ref{eq:clipped-sgd-basic-inequality-nonconvex}) come
from the following 
\begin{align*}
\left\langle \nabla f(x_{t}),\theta_{t}\right\rangle  & =\left\langle \nabla f(x_{t}),\theta_{t}^{u}\right\rangle +\left\langle \nabla f(x_{t}),\theta_{t}^{b}\right\rangle 
\end{align*}
This inner product contains the gradient $\nabla f(x_{t})$ which
also appears in the LHS of (\ref{eq:clipped-sgd-basic-inequality-nonconvex})
as $\left\Vert \nabla f(x_{t})\right\Vert ^{2}$. It is precisely
this matching that allows us to use $\left\Vert \nabla f(x_{t})\right\Vert ^{2}$
to partially absorb the RHS terms. Here, we use Cauchy-Schwarz to
get $\left\langle \nabla f(x_{t}),\theta_{t}^{b}\right\rangle \le\frac{1}{2}\left\Vert \nabla f(x_{t})\right\Vert ^{2}+\frac{1}{2}\left\Vert \theta_{t}^{b}\right\Vert ^{2}$.
With the remaining terms in the RHS of (\ref{eq:clipped-sgd-basic-inequality-nonconvex}),
we can choose the parameters to balance the bound, which turns out
to offer better convergence guarantee. Now, let us return to the convex
case. In the RHS of (\ref{eq:clipped-sgd-basic-inequality-convex}),
we have $\left\langle \theta_{t},x_{t}-x^{*}\right\rangle $, while
the LHS contains the function value gap $f(x_{t})-f^{*}$. Unfortunately,
for general convex functions, we cannot relate the distance $\left\Vert x_{t}-x^{*}\right\Vert $
and the function value gap $f(x_{t})-f^{*}$. As a result, the term
$\left\langle \theta_{t},x_{t}-x^{*}\right\rangle $ prevents us from
achieving a comparable convergence rate as in the non-convex case. 

Next, we show some simple properties of the parameter choice in Lemma
\ref{lem:parameter-property-nonconvex}, where we defer the derivations
to Section \ref{sec:Proof-for-non-convex} of the Appendix.
\begin{lem}
\label{lem:parameter-property-nonconvex}With the choice of $\eta$
and $\lambda$ in Theorem \ref{thm:clipped-sgd-convergence-p-nonconvex},
we have 
\begin{align*}
\frac{1}{L}\left(\frac{\sigma}{\lambda}\right)^{p} & \le\eta\\
\eta & \le\frac{1}{L}\\
\left(\frac{\sigma}{\lambda}\right)^{p}T^{\frac{p}{3p-2}} & \le\frac{1}{32}\\
TL\left(\frac{\sigma}{\lambda}\right)^{p}\lambda^{2}\eta^{2} & \leq\frac{\Delta_{1}}{2048}.
\end{align*}
\end{lem}
We proceed by obtaining a high-probability bound for the RHS of (\ref{eq:sum-nonconvex-basic})
which will help us achieve a high probability bound on $\sum_{t=1}^{T}\left\Vert \nabla f(x_{t})\right\Vert ^{2}$.
Our main result controls the sum over $t$ of $\frac{L\eta^{2}}{2}\left\Vert \theta_{t}\right\Vert ^{2}+\left(L\eta^{2}-\eta\right)\left\langle \nabla f(x_{t}),\theta_{t}^{u}\right\rangle +\frac{\eta}{2}\left\Vert \theta_{t}^{b}\right\Vert ^{2}$
for all $t$. We inductively upper bound them with $\Delta_{1}$ with
high probability. 
\begin{lem}
\label{lem:nonconvex-induction}For $1\le N\le T+1$, let $E_{N}$
be the event that for all $k=1,\dots N$, 
\begin{align*}
\frac{L\eta^{2}}{2}\sum_{t=1}^{k-1}\left\Vert \theta_{t}\right\Vert ^{2}+\left(L\eta^{2}-\eta\right)\sum_{t=1}^{k-1}\left\langle \nabla f(x_{t}),\theta_{t}^{u}\right\rangle +\frac{\eta}{2}\left\Vert \theta_{t}^{b}\right\Vert ^{2} & \le\Delta_{1}.
\end{align*}
Then $E_{N}$ happens with probability at least $1-\frac{(N-1)\delta}{T}$
for each $N\in[T+1]$.
\end{lem}
The key ingredient to prove this lemma is to ensure that the true
gradient $\left\Vert \nabla f(x_{t})\right\Vert ^{2}$ is bounded
by $\Delta_{1}$ with high probability, using the fact that $\left\Vert \nabla f(x_{t})\right\Vert \le\sqrt{2L\Delta_{t}}$
due to the smoothness of $f$ (see Lemma \ref{lem:smooth-prop}),
and our induction to control $\Delta_{t}$. With a bound on $\left\Vert \nabla f(x_{t})\right\Vert $,
we can use the tools from Lemma \ref{lem:Gorbunov-F5-p} and Freedman's
inequality to control the iterates. 

\begin{proof}[Proof of Lemma \ref{lem:nonconvex-induction}]
We will prove by induction on $N$ that $E_{N}$ happens with probability
at least $1-\frac{(N-1)\delta}{T}$. For $N=1$, the event happens
with probability $1$. Suppose that for some $N\le T$, $\Pr\left[E_{N}\right]\ge1-\frac{(N-1)\delta}{T}$.
We will prove that $\Pr\left[E_{N+1}\right]\ge1-\frac{N\delta}{T}$.
We have the LHS of (\ref{eq:clipped-sgd-basic-inequality-nonconvex})
is non-negative, hence for $k\le N$ we have, under the event $E_{N}$
\begin{align*}
\Delta_{k} & \le\Delta_{1}+\left(L\eta^{2}-\eta\right)\sum_{t=1}^{k-1}\left\langle \nabla f(x_{t}),\theta_{t}^{u}\right\rangle +L\eta^{2}\sum_{t=1}^{k-1}\left\Vert \theta_{t}^{u}\right\Vert ^{2}+\frac{3\eta}{2}\left\Vert \theta_{t}^{b}\right\Vert ^{2}\le2\Delta_{1}.
\end{align*}
 Recall that since
\begin{align*}
 & \theta_{t}^{u}=\tn(x_{t})-\E_{t}\left[\tn(x_{t})\right];\quad\mbox{and }\theta_{t}^{b}=\E_{t}\left[\tn(x_{t})\right]-\nabla f(x_{t})
\end{align*}
we have $\theta_{t}=\theta_{t}^{u}+\theta_{t}^{b}$, thus $\left\Vert \theta_{t}\right\Vert ^{2}\le2\left\Vert \theta_{t}^{u}\right\Vert ^{2}+2\left\Vert \theta_{t}^{b}\right\Vert ^{2}$.
We can write
\begin{align*}
 & \left(L\eta^{2}-\eta\right)\sum_{t=1}^{N}\left\langle \nabla f(x_{t}),\theta_{t}^{u}\right\rangle +\frac{3\eta}{2}\sum_{t=1}^{N}\left\Vert \theta_{t}^{b}\right\Vert ^{2}+L\eta^{2}\sum_{t=1}^{N}\left\Vert \theta_{t}^{u}\right\Vert ^{2}\\
\leq & \underbrace{\left(\eta-L\eta^{2}\right)\sum_{t=1}^{N}\left\langle -\nabla f(x_{t}),\theta_{t}^{u}\right\rangle }_{A}+\underbrace{\frac{3\eta}{2}\sum_{t=1}^{N}\left\Vert \theta_{t}^{b}\right\Vert ^{2}}_{B}\\
 & +\underbrace{L\eta^{2}\sum_{t=1}^{N}\left(\left\Vert \theta_{t}^{u}\right\Vert ^{2}-\E_{t}\left[\left\Vert \theta_{t}^{u}\right\Vert ^{2}\right]\right)}_{C}+\underbrace{L\eta^{2}\sum_{t=1}^{N}\E_{t}\left[\left\Vert \theta_{t}^{u}\right\Vert ^{2}\right]}_{D}
\end{align*}
We proceed to bound terms $B,D$ first for they are straightforward
from Lemma \ref{lem:Gorbunov-F5-p}. Then we will bound $A$ and $C$
using Freedman's inequality. First, with probability $1$, we have
$\left\Vert \theta_{t}^{u}\right\Vert \le2\lambda$. Further, when
the event $E_{N}$ happens, and by the smoothness of $f$, we have
\begin{align*}
\left\Vert \nabla f(x_{t})\right\Vert  & \le\sqrt{2L\Delta_{t}}\le\sqrt{4L\Delta_{1}}\le\frac{\lambda}{2}.
\end{align*}
Thus we can apply Lemma \ref{lem:Gorbunov-F5-p} and obtain $\left\Vert \theta_{t}^{b}\right\Vert \le4\sigma^{p}\lambda^{1-p}$
and $\E_{t}\left[\left\Vert \theta_{t}^{u}\right\Vert ^{2}\right]\le16\sigma^{p}\lambda^{2-p}.$

\paragraph{Upperbound for $B$.}

By (\ref{eq:2-p}), when the event $E_{N}$ happens, by Lemma \ref{lem:parameter-property-nonconvex}
\begin{align*}
B & =\frac{3\eta}{2}\left\Vert \theta_{t}^{b}\right\Vert ^{2}\le\frac{3\eta}{2}\sum_{t=1}^{N}16\sigma^{2p}\lambda^{2-2p}=24\sigma^{2p}\lambda^{2-2p}\eta N\\
 & \le24T\left(\frac{\sigma}{\lambda}\right)^{2p}\lambda^{2}\eta\le24TL\left(\frac{\sigma}{\lambda}\right)^{p}\lambda^{2}\eta^{2}\le\frac{3\Delta_{1}}{256}.
\end{align*}

\paragraph{Upperbound for $D$.}

By \ref{eq:3-p}, when the event $E_{N}$ happens,
\begin{align*}
D & =L\eta^{2}\sum_{t=1}^{N}\E_{t}\left[\left\Vert \theta_{t}^{u}\right\Vert ^{2}\right]\le L\eta^{2}\sum_{t=1}^{N}16\sigma^{p}\lambda^{2-p}\\
 & \le16\sigma^{p}\lambda^{2-p}L\eta^{2}N\le16LT\left(\frac{\sigma}{\lambda}\right)^{p}\left(\lambda\eta\right)^{2}\le\frac{\Delta_{1}}{128}.
\end{align*}
To bound $A$ and $C$ we use Freedman's inequality (Theorem \ref{thm:freedman}).
We define, for $t\ge1$, the following random variables
\begin{align*}
Z_{t} & =\begin{cases}
-\nabla f(x_{t}) & \mbox{if }\Delta_{t}\le2\Delta_{1}\\
0 & \mbox{otherwise}.
\end{cases}
\end{align*}
By the smoothness of $f$ and the fact that $f$ is bounded below,
we have (Lemma \ref{lem:smooth-prop}) $\left\Vert \nabla f(x_{t})\right\Vert \le\sqrt{2L\Delta_{t}}.$
Thus with probability $1$, $\left\Vert Z_{t}\right\Vert \le2\sqrt{L\Delta_{1}}$.

\paragraph{Upperbound for $A$.}

Instead of bounding $A=\left(\eta-L\eta^{2}\right)\sum_{t=1}^{N}\left\langle -\nabla f(x_{t}),\theta_{t}^{u}\right\rangle $,
we will bound $A'=\left(\eta-L\eta^{2}\right)\sum_{t=1}^{N}\left\langle Z_{t},\theta_{t}^{u}\right\rangle $.
We check the conditions to apply Freedman's inequality. First $\E_{t}\left[\left(\eta-L\eta^{2}\right)\left\langle Z_{t},\theta_{t}^{u}\right\rangle \right]=0$.
Further, with probability $1$, $\left\Vert \theta_{t}^{u}\right\Vert ^{2}\le2\lambda$,
and $Z_{t}\le2\sqrt{L\Delta_{1}}$, thus$\left|\left(\eta-L\eta^{2}\right)\left\langle Z_{t},\theta_{t}^{u}\right\rangle \right|\le\left(\eta-L\eta^{2}\right)\left\Vert Z_{t}\right\Vert \left\Vert \theta_{t}^{u}\right\Vert \le4\sqrt{L\Delta_{1}}\left(\eta-L\eta^{2}\right)\lambda\le4\sqrt{L\Delta_{1}}\eta\lambda$.
Hence, $\left\{ \left(\eta-L\eta^{2}\right)\left\langle Z_{t},\theta_{t}^{u}\right\rangle \right\} $
is a bounded martingale difference sequence. Therefore, for constant
$a$ and $F$ to be chosen we have
\begin{align*}
 & \Pr\left[\left|\sum_{t=1}^{N}\left(\eta-L\eta^{2}\right)\left\langle Z_{t},\theta_{t}^{u}\right\rangle \right|>a\mbox{ and }\sum_{t=1}^{N}\E_{t}\left[\left(\left(\eta-L\eta^{2}\right)\left\langle Z_{t},\theta_{t}^{u}\right\rangle \right)^{2}\right]\le F\ln\frac{4T}{\delta}\right]\\
 & \le2\exp\left(-\frac{a^{2}}{2F\ln\frac{4T}{\delta}+\frac{8}{3}\sqrt{L\Delta_{1}}\eta\lambda a}\right)
\end{align*}
We choose $a$ such that 
\begin{align*}
2\exp\left(-\frac{a^{2}}{2F\ln\frac{4T}{\delta}+\frac{8}{3}\sqrt{L\Delta_{1}}\eta\lambda a}\right) & =\frac{\delta}{2T}
\end{align*}
which gives 
\begin{align*}
a & =\left(\frac{4}{3}\sqrt{L\Delta_{1}}\eta\lambda+\sqrt{\frac{16L\Delta_{1}\eta^{2}\lambda^{2}}{9}+2F}\right)\ln\frac{4T}{\delta}.
\end{align*}
If we choose $F=64L\Delta_{1}\sigma^{p}\lambda^{2-p}\eta^{2}T$, by
Lemma \ref{lem:parameter-property-nonconvex} we can easily show that
$a\le\frac{7\Delta_{1}}{12}$. Therefore with probability at least
$1-\frac{\delta}{2T}$ we the following event happens 
\begin{align*}
E_{A} & =\Bigg\{\text{either }A'\le\left|\sum_{t=1}^{N}\left(\eta-L\eta^{2}\right)\left\langle Z_{t},\theta_{t}^{u}\right\rangle \right|\le\frac{7\Delta_{1}}{12}\\
 & \text{or }\sum_{t=1}^{N}\E_{t}\left[\left(\left(\eta-L\eta^{2}\right)\left\langle Z_{t},\theta_{t}^{u}\right\rangle \right)^{2}\right]\ge F\ln\frac{4T}{\delta}\Bigg\}
\end{align*}
Also notice that under the event $E_{N}$, we have
\begin{align}
 & \sum_{t=1}^{N}\E_{t}\left[\left(\left(\eta-L\eta^{2}\right)\left\langle Z_{t},\theta_{t}^{u}\right\rangle \right)^{2}\right]\nonumber \\
\le & \eta^{2}\sum_{t=1}^{N}\E_{t}\left[\left\Vert Z_{t}\right\Vert ^{2}\left\Vert \theta_{t}^{u}\right\Vert ^{2}\right]\le4\eta^{2}L\Delta_{1}\sum_{t=1}^{N}\E_{t}\left[\left\Vert \theta_{t}^{u}\right\Vert ^{2}\right]\nonumber \\
\le & 64L\Delta_{1}\sigma^{p}\lambda^{2-p}\eta^{2}N\leq64\Delta_{1}LT\left(\frac{\sigma}{\lambda}\right)^{p}\lambda^{2}\eta^{2}\le F\le F\ln\frac{4T}{\delta}.\label{eq:clipped-sgd-A-variance-bound-p-nonconvex}
\end{align}
Besides under the condition that $E_{N}$ happens, $Z_{t}=-\nabla f(x_{t})$
for all $t\le N$. Therefore, when $E_{N}\cap E_{A}$ happens, we
have $A=A'\le\frac{7\Delta_{1}}{12}.$

\paragraph{Upperbound for $C$.}

We check the conditions to apply Freedman's inequality. First, $\E_{t}\left[L\eta^{2}\left(\left\Vert \theta_{t}^{u}\right\Vert ^{2}-\E_{t}\left[\left\Vert \theta_{t}^{u}\right\Vert ^{2}\right]\right)\right]=0$.
Further, with probability $1$, $\left\Vert \theta_{t}^{u}\right\Vert ^{2}\le2\lambda$,
thus$\left|L\eta^{2}\left(\left\Vert \theta_{t}^{u}\right\Vert ^{2}-\E_{t}\left[\left\Vert \theta_{t}^{u}\right\Vert ^{2}\right]\right)\right|\le L\eta^{2}\left(4\lambda^{2}+4\lambda^{2}\right)=8L\lambda^{2}\eta^{2}$.
Hence $\left\{ L\eta^{2}\left(\left\Vert \theta_{t}^{u}\right\Vert ^{2}-\E_{t}\left[\left\Vert \theta_{t}^{u}\right\Vert ^{2}\right]\right)\right\} $
is a bounded martingale difference sequence. Applying Freedman's inequality
for constants $c$ and $G$ to be chosen, we have
\begin{align*}
 & \Pr\left[\left|L\eta^{2}\sum_{t=1}^{N}\left(\left\Vert \theta_{t}^{u}\right\Vert ^{2}-\E_{t}\left[\left\Vert \theta_{t}^{u}\right\Vert ^{2}\right]\right)\right|>c\mbox{ and }\sum_{t=1}^{N}\E_{t}\left[\left(L\eta^{2}\left(\left\Vert \theta_{t}^{u}\right\Vert ^{2}-\E_{t}\left[\left\Vert \theta_{t}^{u}\right\Vert ^{2}\right]\right)\right)^{2}\right]\le G\ln\frac{4T}{\delta}\right]\\
 & \le2\exp\left(-\frac{c^{2}}{2G\ln\frac{4T}{\delta}+\frac{16}{3}L\lambda^{2}\eta^{2}c}\right)
\end{align*}
We choose $c$ such that 
\begin{align*}
2\exp\left(-\frac{c^{2}}{2G\ln\frac{4T}{\delta}+\frac{16}{3}L\lambda^{2}\eta^{2}c}\right) & =\frac{\delta}{2T}
\end{align*}
which gives 
\begin{align*}
c & =\left(\frac{8}{3}L\lambda^{2}\eta^{2}+\sqrt{\frac{64L^{2}\lambda^{4}\eta^{4}}{9}+2G}\right)\ln\frac{4T}{\delta}
\end{align*}
If we choose $G=256L^{2}\sigma^{p}\lambda^{4-p}\eta^{4}T$, by Lemma
\ref{lem:parameter-property-nonconvex}, a simple calculation shows
that $a\le\frac{7\Delta_{1}}{48}$. This means with probability at
least $1-\frac{\delta}{2T}$, the following event happens 
\begin{align*}
E_{C} & =\Bigg\{\text{either }C\le\left|L\eta^{2}\sum_{t=1}^{N}\left(\left\Vert \theta_{t}^{u}\right\Vert ^{2}-\E_{t}\left[\left\Vert \theta_{t}^{u}\right\Vert ^{2}\right]\right)\right|\le\frac{7\Delta_{1}}{48}\\
 & \text{or }\sum_{t=1}^{N}\E_{t}\left[\left(L\eta^{2}\left(\left\Vert \theta_{t}^{u}\right\Vert ^{2}-\E_{t}\left[\left\Vert \theta_{t}^{u}\right\Vert ^{2}\right]\right)\right)^{2}\right]\ge G\ln\frac{4T}{\delta}\Bigg\}
\end{align*}
Notice that when $G=256L^{2}\sigma^{p}\lambda^{4-p}\eta^{4}T$, under
$E_{N}$ we have 
\begin{align}
 & \sum_{t=1}^{N}\E_{t}\left[\left(L\eta^{2}\left(\left\Vert \theta_{t}^{u}\right\Vert ^{2}-\E_{t}\left[\left\Vert \theta_{t}^{u}\right\Vert ^{2}\right]\right)\right)^{2}\right]\nonumber \\
\le & 8L\lambda^{2}\eta^{2}\sum_{t=1}^{N}\E_{t}\left[\left|L\eta^{2}\left(\left\Vert \theta_{t}^{u}\right\Vert ^{2}-\E_{t}\left[\left\Vert \theta_{t}^{u}\right\Vert ^{2}\right]\right)\right|\right]\le16L^{2}\lambda^{2}\eta^{4}\sum_{t=1}^{N}\E\left[\left\Vert \theta_{t}^{u}\right\Vert ^{2}\right]\nonumber \\
\le & 256L^{2}\sigma^{p}\lambda^{4-p}\eta^{4}N\le G<G\ln\frac{4T}{\delta}.\label{eq:clipped-sgd-C-variance-bound-p-nonconvex}
\end{align}
Therefore, when $E_{N}\cap E_{C}$ happens, we have $C\le\frac{7\Delta_{1}}{48}.$

\paragraph{Combining the bounds}

Finally, we have that the event $E_{N}\cap E_{A}\cap E_{C}$ implies
\begin{align*}
A\le\frac{7\Delta_{1}}{12};\quad B\le\frac{3\Delta_{1}}{256};\quad C\le\frac{7\Delta_{1}}{48};\quad D\le\frac{\Delta_{1}}{128}
\end{align*}
which means 
\begin{align*}
\frac{L\eta^{2}}{2}\sum_{t=1}^{N}\left\Vert \theta_{t}\right\Vert ^{2}+\left(L\eta^{2}-\eta\right)\sum_{t=1}^{N}\left\langle \nabla f(x_{t}),\theta_{t}\right\rangle  & \le A+B+C+D\le\Delta_{1}
\end{align*}
Therefore
\begin{align*}
\Pr\left[E_{N+1}\right] & \ge\Pr\left[E_{N}\cap E_{A}\cap E_{C}\right]=1-\Pr\left[\overline{E}_{N}\cup\overline{E}_{A}\cup\overline{E}_{C}\right]\\
 & \ge1-\frac{(N-1)\delta}{T}-\frac{\delta}{2T}-\frac{\delta}{2T}=1-\frac{N\delta}{T}
\end{align*}
which is what we need to prove.

\end{proof}

\begin{proof}[Proof of Theorem \ref{thm:clipped-sgd-convergence-p-nonconvex}]
To conclude the proof of Theorem \ref{thm:clipped-sgd-convergence-p-nonconvex},
we use Lemma \ref{lem:nonconvex-induction} to see that with probability
at least $1-\delta$,
\begin{align*}
\frac{\eta}{2}\sum_{t=1}^{T}\left\Vert \nabla f(x_{t})\right\Vert ^{2} & \le\Delta_{1}-\Delta_{T+1}+L\eta^{2}\sum_{t=1}^{T}\left\Vert \theta_{t}^{u}\right\Vert ^{2}+\left(L\eta^{2}-\eta\right)\sum_{t=1}^{T}\left\langle \nabla f(x_{t}),\theta_{t}^{u}\right\rangle +\frac{3\eta}{2}\left\Vert \theta_{t}^{b}\right\Vert ^{2}\le2\Delta_{1}
\end{align*}
By the choice of the step size
\begin{align*}
\eta & =\frac{\sqrt{\Delta_{1}}}{8\sqrt{L}\ln\frac{4T}{\delta}}\min\left\{ \left(\frac{8\ln\frac{4T}{\delta}}{\sqrt{L\Delta_{1}}}\right)^{\frac{-1}{p-1}}T^{\frac{-p}{3p-2}}\sigma^{\frac{-p}{p-1}};\frac{T^{\frac{1-p}{3p-2}}}{4\sqrt{L\Delta_{1}}};\frac{T^{\frac{-p}{3p-2}}}{32^{1/p}\sigma}\right\} 
\end{align*}
we have with probability at least $1-\delta$
\begin{align*}
 & \frac{1}{T}\sum_{t=1}^{T}\left\Vert \nabla f(x_{t})\right\Vert ^{2}\le\frac{4\Delta_{1}}{\eta T}\\
\le & 32\sqrt{\Delta_{1}L}\ln\frac{4T}{\delta}\max\left\{ \left(\frac{8\ln\frac{4T}{\delta}}{\sqrt{L\Delta_{1}}}\right)^{\frac{1}{p-1}}T^{\frac{2-2p}{3p-2}}\sigma^{\frac{p}{p-1}};4\sqrt{L\Delta_{1}}T^{\frac{1-2p}{3p-2}};32^{1/p}\sigma T^{\frac{2-2p}{3p-2}}\right\} .
\end{align*}
\end{proof}

\section{Conclusion}
In this work, we show that clipped-SGD, under appropriate parameter choices,  converges optimally with high probability for convex and non-convex objectives.  Our general analysis framework is applicable in both settings. A direct extension of this work would be analyzing the convergence of (accelerated) gradient clipping methods for strongly convex functions and variational inequalities, which has been examined in the works of \citep{gorbunov2020stochastic,gorbunov2022clipped} for noises with bounded variance. For future works, it would be interesting to investigate adaptive methods like Adagrad under heavy-tailed noise, since gradient normalization is built-in for these adaptive methods. 

\bibliographystyle{plainnat}
\bibliography{ref}

\appendix

\section{Proof from Section \ref{sec:Algorithm}\label{sec:Proof-for-sec-algorithm}}

\begin{proof}[Proof of Lemma \ref{lem:Gorbunov-F5-p}]
The proof of this lemma is an extension of Lemma F5 from \citet{gorbunov2020stochastic}.

\paragraph{For (\ref{eq:1-p})}

By definition, $\left\Vert \tn(x_{t})\right\Vert \le\lambda_{t}$

\begin{align*}
\left\Vert \theta_{t}^{u}\right\Vert  & =\left\Vert \tn(x_{t})-\E_{t}\left[\tn(x_{t})\right]\right\Vert \\
 & \le\left\Vert \tn(x_{t})\right\Vert +\left\Vert \E_{t}\left[\tn(x_{t})\right]\right\Vert \le2\lambda_{t}.
\end{align*}

\paragraph{For (\ref{eq:2-p})}

Let 
\begin{align*}
\chi_{t} & =\indicator\left\{ \left\Vert \hn(x_{t})\right\Vert >\lambda_{t}\right\} \\
\mu_{t} & =\indicator\left\{ \left\Vert \hn(x_{t})-\nabla f(x_{t})\right\Vert >\frac{\lambda_{t}}{2}\right\} 
\end{align*}
By the assumption $\left\Vert \nabla f(x_{t})\right\Vert \le\frac{\lambda_{t}}{2}$,
we have 
\begin{align*}
\left\Vert \hn(x_{t})\right\Vert  & \le\left\Vert \hn(x_{t})-\nabla f(x_{t})\right\Vert +\left\Vert \nabla f(x_{t})\right\Vert \\
 & \le\left\Vert \hn(x_{t})-\nabla f(x_{t})\right\Vert +\frac{\lambda_{t}}{2}
\end{align*}
Hence $\chi_{t}\le\mu_{t}$. We write 
\begin{align}
\tn(x_{t}) & =\hn(x_{t})(1-\chi_{t})+\frac{\lambda_{t}}{\left\Vert \hn(x_{t})\right\Vert }\hn(x_{t})\chi_{t}\label{eq:clipped-gradient-rewrite-1-p}\\
 & =\hn(x_{t})+\left(\frac{\lambda_{t}}{\left\Vert \hn(x_{t})\right\Vert }-1\right)\hn(x_{t})\chi_{t}\label{eq:clipped-gradient-rewrite-2-p}
\end{align}
Hence 
\begin{align*}
\left\Vert \theta_{t}^{b}\right\Vert  & =\left\Vert \E_{t}\left[\tn(x_{t})\right]-\nabla f(x_{t})\right\Vert \\
 & =\left\Vert \E_{t}\left[\left(\frac{\lambda_{t}}{\left\Vert \hn(x_{t})\right\Vert }-1\right)\hn(x_{t})\chi_{t}\right]\right\Vert \\
 & \le\E_{t}\left[\left\Vert \hn(x_{t})\right\Vert \left|\frac{\lambda_{t}}{\left\Vert \hn(x_{t})\right\Vert }-1\right|\chi_{t}\right]\\
 & =\E_{t}\left[\left\Vert \hn(x_{t})\right\Vert \left(1-\frac{\lambda_{t}}{\left\Vert \hn(x_{t})\right\Vert }\right)\chi_{t}\right]\\
 & \le\E_{t}\left[\left\Vert \hn(x_{t})\right\Vert \chi_{t}\right]\\
 & \le\E_{t}\left[\left\Vert \hn(x_{t})\right\Vert \mu_{t}\right]\\
 & \overset{(*)}{\le}\E_{t}\left[\left\Vert \hn(x_{t})-\nabla f(x_{t})\right\Vert \mu_{t}\right]+\left\Vert \nabla f(x_{t})\right\Vert \E_{t}\left[\mu_{t}\right]\\
 & \overset{(**)}{\le}\E_{t}\left[\left\Vert \hn(x_{t})-\nabla f(x_{t})\right\Vert ^{p}\right]^{\frac{1}{p}}\E_{t}\left[\mu_{t}^{\frac{p}{p-1}}\right]^{\frac{p-1}{p}}+\left\Vert \nabla f(x_{t})\right\Vert \E_{t}\left[\mu_{t}\right]\\
 & \le\sigma\E_{t}\left[\mu_{t}\right]^{\frac{p-1}{p}}+\frac{\lambda_{t}}{2}\E_{t}\left[\mu_{t}\right]
\end{align*}
For $(*)$ we use the triangle inequality; for $(**)$ we use Holder's
inequality. To bound $\E_{t}\left[\mu_{t}\right]$, by Markov's inequality
\begin{align*}
\E_{t}\left[\mu_{t}\right] & =\Pr\left[\left\Vert \hn(x_{t})-\nabla f(x_{t})\right\Vert >\frac{\lambda_{t}}{2}\mid\F_{t}\right]\\
 & \le\frac{\E_{t}\left[\left\Vert \hn(x_{t})-\nabla f(x_{t})\right\Vert ^{p}\right]}{(\lambda_{t}/2)^{p}}\\
 & \le\frac{2^{p}\sigma^{p}}{\lambda_{t}^{p}}
\end{align*}
Hence, we have 
\begin{align*}
\left\Vert \E_{t}\left[\tn(x_{t})\right]-\nabla f(x_{t})\right\Vert  & \le\sigma\frac{2^{p-1}\sigma^{p-1}}{\lambda_{t}^{p-1}}+\frac{\lambda_{t}}{2}\frac{2^{p}\sigma^{p}}{\lambda_{t}^{p}}\\
 & =2^{p}\sigma^{p}\lambda_{t}^{1-p}\le4\sigma^{p}\lambda_{t}^{1-p}.
\end{align*}

\paragraph{For (\ref{eq:3-p})}

First we will prove that 
\begin{align*}
\E_{t}\left[\left\Vert \tn(x_{t})-\nabla f(x_{t})\right\Vert ^{2}\right] & \le16\sigma^{p}\lambda_{t}^{2-p}.
\end{align*}
From \ref{eq:clipped-gradient-rewrite-1-p}

\begin{align*}
\sqrt{\E_{t}\left[\left\Vert \tn(x_{t})-\nabla f(x_{t})\right\Vert ^{2}\right]} & =\sqrt{\E_{t}\left[\left\Vert \left(\frac{\lambda_{t}\hn(x_{t})}{\left\Vert \hn(x_{t})\right\Vert }-\nabla f(x_{t})\right)\chi_{t}+\left(\hn(x_{t})-\nabla f(x_{t})\right)\right\Vert ^{2}(1-\chi_{t})\right]}\\
 & \overset{(a)}{\le}\sqrt{\E_{t}\left[\left\Vert \frac{\lambda_{t}\hn(x_{t})}{\left\Vert \hn(x_{t})\right\Vert }-\nabla f(x_{t})\right\Vert ^{2}\chi_{t}^{2}\right]}+\sqrt{\E_{t}\left[\left\Vert \hn(x_{t})-\nabla f(x_{t})\right\Vert ^{2}(1-\chi_{t})^{2}\right]}\\
 & \overset{(b)}{\le}\sqrt{\E_{t}\left[\left(\left\Vert \frac{\lambda_{t}\hn(x_{t})}{\left\Vert \hn(x_{t})\right\Vert }\right\Vert +\left\Vert \nabla f(x_{t})\right\Vert \right)^{2}\mu_{t}\right]}\\
 & \quad+\sqrt{\E_{t}\left[\left\Vert \hn(x_{t})-\nabla f(x_{t})\right\Vert ^{p}\left(\left\Vert \hn(x_{t})\right\Vert +\left\Vert \nabla f(x_{t})\right\Vert \right)^{2-p}(1-\chi_{t})\right]}\\
 & \overset{(c)}{\le}\frac{3}{2}\lambda_{t}\sqrt{\E_{t}\left[\mu_{t}\right]}+\left(\frac{3}{2}\lambda\right)^{\frac{2-p}{2}}\sqrt{\E_{t}\left[\left\Vert \hn(x_{t})-\nabla f(x_{t})\right\Vert ^{p}\right]}\\
 & \le\frac{3}{2}\times2^{\frac{p}{2}}\lambda_{t}^{\frac{2-p}{2}}\sigma^{\frac{p}{2}}+\left(\frac{3}{2}\right)^{\frac{2-p}{2}}\lambda_{t}^{\frac{2-p}{2}}\sigma^{\frac{p}{2}}\\
 & =\frac{3}{2}\left(2^{\frac{p}{2}}+\left(\frac{3}{2}\right)^{-\frac{p}{2}}\right)\lambda_{t}^{\frac{2-p}{2}}\sigma^{\frac{p}{2}}\\
 & \overset{(d)}{\le}4\lambda_{t}^{\frac{2-p}{2}}\sigma^{\frac{p}{2}}
\end{align*}
For $(a)$ we use $\sqrt{a+b}\le\sqrt{a}+\sqrt{b}$. For $(b)$ we
use $\left\Vert a+b\right\Vert \le\left\Vert a\right\Vert +\left\Vert b\right\Vert $.
For $(c)$ we use $\left\Vert \nabla f(x_{t})\right\Vert \le\frac{\lambda_{t}}{2}$
and
\[
\left(\left\Vert \hn(x_{t})\right\Vert +\left\Vert \nabla f(x_{t})\right\Vert \right)^{2-p}(1-\chi_{t})\le\left(\frac{3}{2}\lambda_{t}\right)^{2-p}
\]
 since $\chi_{t}=\indicator\left\{ \left\Vert \hn(x_{t})\right\Vert >\lambda\right\} $.
For $(d)$ we use that 
\[
\max_{p\in(1,2]}2^{\frac{p}{2}}+\left(\frac{3}{2}\right)^{-\frac{p}{2}}=\frac{8}{3}\mbox{ at }p=2
\]
which gives us 
\begin{align*}
\E_{t}\left[\left\Vert \tn(x_{t})-\nabla f(x_{t})\right\Vert ^{2}\right] & \le16\sigma^{p}\lambda_{t}^{2-p}.
\end{align*}
Since $\E_{t}\left[\tn(x_{t})\right]$ is the minimizer of $\E_{t}\left[\left\Vert \tn(x_{t})-x\right\Vert ^{2}\right]$

\begin{align*}
\E_{t}\left[\left\Vert \tn(x_{t})-\E_{t}\left[\tn(x_{t})\right]\right\Vert ^{2}\right] & \le\E_{t}\left[\left\Vert \tn(x_{t})-\nabla f(x_{t})\right\Vert ^{2}\right]\le16\sigma^{p}\lambda_{t}^{2-p}.
\end{align*}
\end{proof}

\section{Proof from Section \ref{sec:Convex}\label{sec:appendix-convex-pfs}}

\begin{proof}[Proof of Lemma \ref{lem:convex-basic-analysis}]
We start from the convexity of $f$ and use the update $x_{t+1}=x_{t}-\frac{1}{\eta_{t}}\tn(x_{t})$
to obtain
\begin{align*}
f(x_{t})-f^{*} & \le\left\langle \nabla f(x_{t}),x_{t}-x^{*}\right\rangle \\
 & =\left\langle \tn(x_{t}),x_{t}-x^{*}\right\rangle -\left\langle \tn(x_{t})-\nabla f(x_{t}),x_{t}-x^{*}\right\rangle \\
 & =\frac{1}{\eta_{t}}\left\langle x_{t}-x_{t+1},x_{t}-x^{*}\right\rangle -\left\langle \tn(x_{t})-\nabla f(x_{t}),x_{t}-x^{*}\right\rangle \\
 & =\frac{1}{2}\left\Vert x_{t}-x_{t+1}\right\Vert ^{2}+\frac{1}{2\eta_{t}}\left\Vert x_{t}-x^{*}\right\Vert ^{2}-\frac{1}{2\eta_{t}}\left\Vert x_{t+1}-x^{*}\right\Vert ^{2}-\left\langle \theta_{t},x_{t}-x^{*}\right\rangle \\
 & =\frac{\eta_{t}}{2}\left\Vert \tn(x_{t})\right\Vert ^{2}+\frac{1}{2\eta_{t}}\left\Vert x_{t}-x^{*}\right\Vert ^{2}-\frac{1}{2\eta_{t}}\left\Vert x_{t+1}-x^{*}\right\Vert ^{2}-\left\langle \theta_{t},x_{t}-x^{*}\right\rangle \\
 & \overset{(a)}{\le}\eta_{t}\left\Vert \nabla f(x_{t})\right\Vert ^{2}+\eta_{t}\left\Vert \theta_{t}\right\Vert ^{2}+\frac{1}{2\eta_{t}}\left\Vert x_{t}-x^{*}\right\Vert ^{2}-\frac{1}{2\eta_{t}}\left\Vert x_{t+1}-x^{*}\right\Vert ^{2}-\left\langle \theta_{t},x_{t}-x^{*}\right\rangle \\
 & \overset{(b)}{\le}2L\eta_{t}(f(x_{t})-f^{*})+\eta_{t}\left\Vert \theta_{t}\right\Vert ^{2}+\frac{1}{2\eta_{t}}\left\Vert x_{t}-x^{*}\right\Vert ^{2}-\frac{1}{2\eta_{t}}\left\Vert x_{t+1}-x^{*}\right\Vert ^{2}-\left\langle \theta_{t},x_{t}-x^{*}\right\rangle 
\end{align*}
For $(a)$, we use $\left\Vert \tn(x_{t})\right\Vert ^{2}=\left\Vert \nabla f(x_{t})+\theta_{t}\right\Vert ^{2}\le2\left\Vert \nabla f(x_{t})\right\Vert ^{2}+2\left\Vert \theta_{t}\right\Vert ^{2}$.
For $(b)$, we use the smoothness of $f$ from Lemma \ref{lem:smooth-prop}.
Rearranging the terms, we obtain 
\begin{align*}
(2\eta_{t}-4\eta_{t}^{2}L)\Delta_{t} & \le\left\Vert x_{t}-x^{*}\right\Vert ^{2}-\left\Vert x_{t+1}-x^{*}\right\Vert ^{2}+2\eta_{t}^{2}\left\Vert \theta_{t}\right\Vert ^{2}-2\eta_{t}\left\langle \theta_{t},x_{t}-x^{*}\right\rangle .
\end{align*}
Since $\eta\le\frac{1}{4L}$ we have 
\begin{align*}
\eta_{t}\Delta_{t} & \le\left\Vert x_{t}-x^{*}\right\Vert ^{2}-\left\Vert x_{t+1}-x^{*}\right\Vert ^{2}+2\eta_{t}^{2}\left\Vert \theta_{t}\right\Vert ^{2}-2\eta_{t}\left\langle \theta_{t},x_{t}-x^{*}\right\rangle .
\end{align*}
\end{proof}

\begin{proof}[Proof of Theorem \ref{thm:clipped-sgd-convergence-p-convex}]
For $1\le N\le T+1$, let $E_{N}$ be the event that for all $k=1,\dots N$,
\begin{align*}
2\eta^{2}\sum_{t=1}^{k-1}\left\Vert \theta_{t}\right\Vert ^{2}-2\eta\sum_{t=1}^{k-1}\left\langle \theta_{t},x_{t}-x^{*}\right\rangle  & \le R_{1}^{2}.
\end{align*}
We will prove by induction on $N$ that $E_{N}$ happens with probability
at least $1-\frac{(N-1)\delta}{T}$. For $N=1$, the event happens
with probability $1$. Suppose that for some $N\le T$, $\Pr\left[E_{N}\right]\ge1-\frac{(N-1)\delta}{T}$.
We will prove that $\Pr\left[E_{N+1}\right]\ge1-\frac{N\delta}{T}$.
We have the LHS of \ref{eq:clipped-sgd-basic-inequality-nonconvex}
is non-negative, hence for $k\le N$ we have, under the event $E_{N}$
\begin{align*}
R_{k}^{2} & \le R_{1}^{2}+2\eta^{2}\sum_{t=1}^{k-1}\left\Vert \theta_{t}\right\Vert ^{2}-2\eta\sum_{t=1}^{k-1}\left\langle \theta_{t},x_{t}-x^{*}\right\rangle \le2R_{1}^{2}
\end{align*}
We define the following random variables
\begin{align*}
Z_{t} & =\begin{cases}
x^{*}-x_{t} & \mbox{if }R_{t}^{2}\le2R_{1}^{2}\\
0 & \mbox{otherwise}
\end{cases}
\end{align*}
Note that $\left\Vert Z_{t}\right\Vert \le\sqrt{2}R_{1}\le2R_{1}$
with probability $1$. Recall 
\begin{align*}
\theta_{t}^{u} & =\tn(x_{t})-\E_{t}\left[\tn(x_{t})\right]\\
\theta_{t}^{b} & =\E_{t}\left[\tn(x_{t})\right]-\nabla f(x_{t})
\end{align*}
and $\theta_{t}=\theta_{t}^{u}+\theta_{t}^{b}$, thus $\left\Vert \theta_{t}\right\Vert ^{2}\le2\left\Vert \theta_{t}^{u}\right\Vert ^{2}+2\left\Vert \theta_{t}^{b}\right\Vert ^{2}$.
Also notice that when $E_{N}$ happens we have $Z_{t}=-\nabla f(x_{t})$
for every $t\le N$. Hence $E_{N}$ implies
\begin{align*}
 & 2\eta^{2}\sum_{t=1}^{N}\left\Vert \theta_{t}\right\Vert ^{2}-2\eta\sum_{t=1}^{N}\left\langle x_{t}-x^{*},\theta_{t}\right\rangle \\
\leq & \underbrace{2\eta\sum_{t=1}^{N}\left\langle x^{*}-x_{t},\theta_{t}^{u}\right\rangle }_{A}+\underbrace{2\eta\sum_{t=1}^{N}\left\langle x^{*}-x_{t},\theta_{t}^{b}\right\rangle }_{B}\\
 & +\underbrace{4\eta^{2}\sum_{t=1}^{N}\left(\left\Vert \theta_{t}^{u}\right\Vert ^{2}-\E_{t}\left[\left\Vert \theta_{t}^{u}\right\Vert ^{2}\right]\right)}_{C}+\underbrace{4\eta^{2}\sum_{t=1}^{N}\E_{t}\left[\left\Vert \theta_{t}^{u}\right\Vert ^{2}\right]}_{D}+\underbrace{4\eta^{2}\sum_{t=1}^{N}\left\Vert \theta_{t}^{b}\right\Vert ^{2}}_{E}
\end{align*}
We proceed to bound terms $B,D,E$ first for they are straightforward
from Lemma \ref{lem:Gorbunov-F5-p}, then we will bound $A$ and $C$.

First, with probability $1$, we have 
\begin{align*}
\left\Vert \theta_{t}^{u}\right\Vert  & \le2\lambda
\end{align*}
 Further, when the event $E_{N}$ happens, and by the smoothness of
$f$, we have 
\begin{align*}
\left\Vert \nabla f(x_{t})\right\Vert  & =\left\Vert \nabla f(x_{t})-\nabla f(x^{*})\right\Vert \le L\left\Vert x_{t}-x^{*}\right\Vert \le LR_{t}\le\sqrt{2}LR_{1}\le\frac{\lambda}{2}
\end{align*}
Thus we can apply Lemma \ref{lem:Gorbunov-F5-p} to obtain $\left\Vert \theta_{t}^{b}\right\Vert \le4\sigma^{p}\lambda^{1-p}$
and $\E_{t}\left[\left\Vert \theta_{t}^{u}\right\Vert ^{2}\right]\le16\sigma^{p}\lambda^{2-p}.$

\paragraph{Upperbound for $B$.}

When the event $E_{N}$ happens, $Z_{t}=-\nabla f(x_{t})$ for all
$t\le N$. Hence, by (\ref{eq:2-p}),
\begin{align*}
B & =2\eta\sum_{t=1}^{N}\left\langle Z_{t},\theta_{t}^{b}\right\rangle \le2\eta\sum_{t=1}^{N}\left\Vert Z_{t}\right\Vert \left\Vert \theta_{t}^{b}\right\Vert =2\eta\sum_{t=1}^{N}2R_{1}4\sigma^{p}\lambda^{1-p}\\
 & \le16R_{1}T\sigma^{p}\lambda^{1-p}\eta\le16R_{1}\left(\frac{\sigma}{\lambda}\right)^{p}T\frac{R_{1}}{16\ln\frac{4T}{\delta}}\le\frac{R_{1}^{2}}{16}.
\end{align*}

\paragraph{Upperbound for $D$.}

When the event $E_{N}$ happens, by (\ref{eq:3-p}),
\begin{align*}
D & =4\eta^{2}\sum_{t=1}^{N}\E_{t}\left[\left\Vert \theta_{t}^{u}\right\Vert ^{2}\right]\le4\eta^{2}\sum_{t=1}^{N}16\sigma^{p}\lambda^{2-p}\le64\sigma^{p}\lambda^{2-p}\eta^{2}T\\
 & \le64\left(\frac{\sigma}{\lambda}\right)^{p}T\left(\frac{R_{1}}{16\ln\frac{4T}{\delta}}\right)^{2}\le\frac{R_{1}^{2}}{64}.
\end{align*}

\paragraph{Upperbound for $E$.}

When the event $E_{N}$ happens, by (\ref{eq:2-p}),
\begin{align*}
E & =4\eta^{2}\sum_{0}^{N}\left\Vert \theta_{k}^{b}\right\Vert ^{2}\le4\eta^{2}\sum_{t=1}^{N}\left(4\sigma^{p}\lambda^{1-p}\right)^{2}=64\sigma^{2p}\lambda^{2-2p}\eta^{2}T\\
 & \le64\left(\frac{\sigma}{\lambda}\right)^{2p}T\left(\frac{R_{1}}{16\ln\frac{4T}{\delta}}\right)^{2}\le\frac{R_{1}^{2}}{1024}.
\end{align*}

To bound $A$ and $C$ we use Freedman's inequality (\ref{thm:freedman}).

\paragraph{Upperbound for $A$.}

Instead of bounding $A=\sum_{t=1}^{N}2\eta\left\langle x^{*}-x_{t},\theta_{t}^{u}\right\rangle $,
we will bound $A'=\sum_{t=1}^{N}2\eta\left\langle Z_{t},\theta_{t}^{u}\right\rangle $.
We check the conditions to apply Freedman's inequality. First $\E_{t}\left[2\eta\left\langle Z_{t},\theta_{t}^{u}\right\rangle \right]=0$.
Further, with probability $1$, $\left\Vert \theta_{t}^{u}\right\Vert \le2\lambda$
and $Z_{t}\le2R_{1}$, thus $\left|2\eta\left\langle Z_{t},\theta_{t}^{u}\right\rangle \right|\le2\eta\left\Vert Z_{t}\right\Vert \left\Vert \theta_{t}^{u}\right\Vert \le4R_{1}\eta\lambda$.
Hence, the sequence $\left\{ 2\eta\left\langle Z_{t},\theta_{t}^{u}\right\rangle \right\} $
is a bounded martingale difference sequence. Therefore, for constant
$a$ and $F$ to be chosen we have
\begin{align*}
 & \Pr\left[\left|\sum_{t=1}^{N}2\eta\left\langle Z_{t},\theta_{t}^{u}\right\rangle \right|>a\mbox{ and }\sum_{t=1}^{N}\E_{t}\left[\left(2\eta\left\langle Z_{t},\theta_{t}^{u}\right\rangle \right)^{2}\right]\le F\ln\frac{4T}{\delta}\right]\\
 & \le2\exp\left(-\frac{a^{2}}{2F\ln\frac{4T}{\delta}+\frac{8}{3}R_{1}\eta\lambda a}\right)
\end{align*}
We choose $b$ such that 
\begin{align*}
2\exp\left(-\frac{a^{2}}{2F\ln\frac{4T}{\delta}+\frac{8}{3}R_{1}\eta\lambda a}\right) & =\frac{\delta}{2T}
\end{align*}
which gives 
\begin{align*}
a & =\left(\frac{4}{3}R_{1}\eta\lambda+\sqrt{\frac{16R_{1}^{2}\eta^{2}\lambda^{2}}{9}+2F}\right)\ln\frac{4T}{\delta}
\end{align*}
If we choose $F=128R_{1}^{2}\sigma^{p}\lambda^{2-p}\eta^{2}T$ we
have
\begin{align*}
a & =\left(\frac{4}{3}R_{1}\eta\lambda+\sqrt{\frac{16R_{1}^{2}\eta^{2}\lambda^{2}}{9}+256R_{1}^{2}\sigma^{p}\lambda^{2-p}\eta^{2}T}\right)\ln\frac{4T}{\delta}\\
 & \le\left(\frac{4}{3}+\sqrt{\frac{16}{9}+256\left(\frac{\sigma}{\lambda}\right)^{p}T}\right)R_{1}\eta\lambda\ln\frac{4T}{\delta}\\
 & \le\left(\frac{4}{3}+\frac{4}{3}+4\right)\frac{R_{1}^{2}}{16}=\frac{5R_{1}^{2}}{12}
\end{align*}
Therefore with probability at least $1-\frac{\delta}{2T}$ we the
following event happens 
\begin{align*}
E_{A} & =\Bigg\{\text{either }A'\le\left|\sum_{t=1}^{N}2\eta\left\langle Z_{t},\theta_{t}^{u}\right\rangle \right|\le\frac{5R_{1}^{2}}{12}\\
 & \text{or }\sum_{t=1}^{N}\E_{t}\left[\left(2\eta\left\langle Z_{t},\theta_{t}^{u}\right\rangle \right)^{2}\right]\ge F\ln\frac{4T}{\delta}\Bigg\}
\end{align*}
Also notice that under the event $E_{N}$ we have 
\begin{align}
\sum_{t=1}^{N}\E_{t}\left[\left(2\eta\left\langle Z_{t},\theta_{t}^{u}\right\rangle \right)^{2}\right] & \le4\eta^{2}\sum_{t=1}^{N}\E_{t}\left[\left\Vert Z_{t}\right\Vert ^{2}\left\Vert \theta_{t}^{u}\right\Vert ^{2}\right]\le8\eta^{2}R_{1}^{2}\sum_{t=1}^{N}\E_{t}\left[\left\Vert \theta_{t}^{u}\right\Vert ^{2}\right]\nonumber \\
 & \le128R_{1}^{2}\sigma^{p}\lambda^{2-p}\eta^{2}N\le F<F\ln\frac{4T}{\delta}.\label{eq:clipped-sgd-A-variance-bound-p-convex}
\end{align}
Besides under the condition that $E_{N}$ happens, $Z_{t}=-\nabla f(x_{t})$
for all $t\le N$. Therefore, when $E_{N}\cap E_{A}$ happens, we
have $A=A'\le\frac{5R_{1}^{2}}{12}.$

\paragraph{Upperbound for $C$.}

We check the conditions to apply Freedman's inequality. First, $\E_{t}\left[4\eta^{2}\left(\left\Vert \theta_{t}^{u}\right\Vert ^{2}-\E_{t}\left[\left\Vert \theta_{t}^{u}\right\Vert ^{2}\right]\right)\right]=0$.
Further, with probability $1$, $\left\Vert \theta_{t}^{u}\right\Vert \le2\lambda$,
thus $\left|4\eta^{2}\left(\left\Vert \theta_{t}^{u}\right\Vert ^{2}-\E_{t}\left[\left\Vert \theta_{t}^{u}\right\Vert ^{2}\right]\right)\right|\le4\eta^{2}\left(4\lambda^{2}+4\lambda^{2}\right)=32\lambda^{2}\eta^{2}$.
Hence, the sequence $\left\{ 4\eta^{2}\left(\left\Vert \theta_{t}^{u}\right\Vert ^{2}-\E_{t}\left[\left\Vert \theta_{t}^{u}\right\Vert ^{2}\right]\right)\right\} $
is a bounded martingale difference sequence. Applying Freedman's inequality
for constants $c$ and $G$ to be chosen, we have
\begin{align*}
 & \Pr\left[\left|4\eta^{2}\sum_{t=1}^{N}\left(\left\Vert \theta_{t}^{u}\right\Vert ^{2}-\E_{t}\left[\left\Vert \theta_{t}^{u}\right\Vert ^{2}\right]\right)\right|>c\mbox{ and }\sum_{t=1}^{N}\E_{t}\left[\left(4\eta^{2}\left(\left\Vert \theta_{t}^{u}\right\Vert ^{2}-\E_{\xi_{t}}\left[\left\Vert \theta_{t}^{u}\right\Vert ^{2}\right]\right)\right)^{2}\right]\le G\ln\frac{4T}{\delta}\right]\\
 & \le2\exp\left(-\frac{c^{2}}{2G\ln\frac{4T}{\delta}+\frac{64}{3}\lambda^{2}\eta^{2}c}\right)
\end{align*}
We choose $c$ such that 
\begin{align*}
2\exp\left(-\frac{c^{2}}{2G\ln\frac{4T}{\delta}+\frac{64}{3}\lambda^{2}\eta^{2}c}\right) & =\frac{\delta}{2T}
\end{align*}
which gives 
\begin{align*}
c & =\left(\frac{32}{3}\lambda^{2}\eta^{2}+\sqrt{\frac{1024\lambda^{4}\eta^{4}}{9}+2G}\right)\ln\frac{4T}{\delta}
\end{align*}
if we choose $G=1024\sigma^{p}\lambda^{4-p}\eta^{4}T$ then 
\begin{align*}
c & =\left(\frac{32}{3}\lambda^{2}\eta^{2}+\sqrt{\frac{1024\lambda^{4}\eta^{4}}{9}+2048\sigma^{p}\lambda^{4-p}\eta^{4}T}\right)\ln\frac{4T}{\delta}\\
 & =\left(\frac{32}{3}+\sqrt{\frac{1024}{9}+2048\left(\frac{\sigma}{\lambda}\right)^{p}T}\right)\lambda^{2}\eta^{2}\ln\frac{4T}{\delta}\\
 & \le\left(\frac{32}{3}+\frac{32}{3}+12\right)\left(\frac{R_{1}}{16\ln\frac{4T}{\delta}}\right)^{2}\ln\frac{4T}{\delta}\le\frac{25R_{1}^{2}}{192}
\end{align*}
This means with probability at least $1-\frac{\delta}{2T}$, the following
event happens 
\begin{align*}
E_{C} & =\Bigg\{\text{either }C\le\left|4\eta^{2}\sum_{t=1}^{N}\left(\left\Vert \theta_{t}^{u}\right\Vert ^{2}-\E_{t}\left[\left\Vert \theta_{t}^{u}\right\Vert ^{2}\right]\right)\right|\le\frac{25R_{1}^{2}}{192}\\
 & \text{or }\sum_{t=1}^{N}\E_{t}\left[\left(4\eta^{2}\left(\left\Vert \theta_{t}^{u}\right\Vert ^{2}-\E_{t}\left[\left\Vert \theta_{t}^{u}\right\Vert ^{2}\right]\right)\right)^{2}\right]\ge G\ln\frac{4T}{\delta}\Bigg\}
\end{align*}
Notice that when $G=1024\sigma^{p}\lambda^{4-p}\eta^{4}T$, under
$E_{N}$ we have 
\begin{align}
 & \sum_{t=1}^{N}\E_{t}\left[\left(4\eta^{2}\left(\left\Vert \theta_{t}^{u}\right\Vert ^{2}-\E_{t}\left[\left\Vert \theta_{t}^{u}\right\Vert ^{2}\right]\right)\right)^{2}\right]\nonumber \\
\le & 32\lambda^{2}\eta^{2}\sum_{t=1}^{N}\E_{t}\left[\left|4\eta^{2}\left(\left\Vert \theta_{t}^{u}\right\Vert ^{2}-\E_{t}\left[\left\Vert \theta_{t}^{u}\right\Vert ^{2}\right]\right)\right|\right]\nonumber \\
\le & 64\lambda^{2}\eta^{4}\sum_{t=1}^{N}\E\left[\left\Vert \theta_{t}^{u}\right\Vert ^{2}\right]\le1024\sigma^{p}\lambda^{4-p}\eta^{4}N\le G<G\ln\frac{4T}{\delta}.\label{eq:clipped-sgd-C-variance-bound-p-convex}
\end{align}
This means, when $E_{N}\cap E_{C}$ happens, we must have $C\le\frac{7R_{1}^{2}}{96}$.

\paragraph{Combining the bounds}

Due to the observations in (\ref{eq:clipped-sgd-A-variance-bound-p-convex})
and (\ref{eq:clipped-sgd-C-variance-bound-p-convex}), we have that
the event $E_{N}\cap E_{A}\cap E_{C}$ implies 
\begin{align*}
A\le\frac{5R^{2}}{12};\quad B\le\frac{R^{2}}{16};\quad C\le\frac{25R_{1}^{2}}{192};\quad D\le\frac{R_{1}^{2}}{64};\quad E\le\frac{R_{1}^{2}}{1024}
\end{align*}
which means $E_{N+1}$ happens because
\begin{align*}
2\eta^{2}\sum_{t=1}^{N}\left\Vert \theta_{t}\right\Vert ^{2}-2\eta\sum_{t=1}^{N}\left\langle \theta_{t},x_{t}-x^{*}\right\rangle  & \le A+B+C+D+E\le R_{1}^{2}
\end{align*}
Therefore
\begin{align*}
\Pr\left[E_{N+1}\right] & \ge\Pr\left[E_{N}\cap E_{A}\cap E_{C}\right]=1-\Pr\left[\overline{E}_{N}\cup\overline{E}_{A}\cup\overline{E}_{C}\right]\\
 & \ge1-\frac{(N-1)\delta}{T}-\frac{\delta}{2T}-\frac{\delta}{2T}=1-\frac{N\delta}{T}
\end{align*}
which is what we need to prove.

To conclude the proof of Theorem \ref{thm:clipped-sgd-convergence-p-convex},
we have with probability at least $1-\delta$ we have 
\begin{align*}
\eta\sum_{t=1}^{T}\Delta_{t} & \le R_{1}^{2}-R_{T+1}^{2}+2\eta^{2}\sum_{t=1}^{T}\left\Vert \theta_{t}\right\Vert ^{2}-2\eta\sum_{t=1}^{T}\left\langle \theta_{t},x_{t}-x^{*}\right\rangle \le2R_{1}^{2}
\end{align*}
Finally since 
\begin{align*}
\eta & =\frac{R_{1}}{16\ln\frac{4T}{\delta}}\min\left\{ (16T)^{-1/p}\sigma^{-1};(\sqrt{2}LR_{1})^{-1}\right\} 
\end{align*}
we have with probability at least $1-\delta$
\begin{align*}
\frac{1}{T}\sum_{t=1}^{T}\Delta_{t} & \le\frac{2R_{1}^{2}}{\eta T}\le32R_{1}\ln\frac{4T}{\delta}\max\left\{ 16^{1/p}T^{\frac{1-p}{p}};\sqrt{2}LR_{1}T^{-1}\right\} .
\end{align*}
\end{proof}

\section{Proof from Section \ref{sec:Non-convex}\label{sec:Proof-for-non-convex}}

\begin{proof}[Proof of Lemma \ref{lem:nonconvex-basic-analysis}]
By the smoothness of $f$ and the update $x_{t+1}=x_{t}-\frac{1}{\eta_{t}}\tn(x_{t})$
we have

\begin{align*}
 & f(x_{t+1})-f(x_{t})\\
\le & \left\langle \nabla f(x_{t}),x_{t+1}-x_{t}\right\rangle +\frac{L}{2}\left\Vert x_{t+1}-x_{t}\right\Vert ^{2}\\
= & -\eta_{t}\left\langle \nabla f(x_{t}),\tn(x_{t})\right\rangle +\frac{L\eta_{t}^{2}}{2}\left\Vert \tn(x_{t})\right\Vert ^{2}\\
= & -\eta_{t}\left\langle \nabla f(x_{t}),\theta_{t}+\nabla f(x_{t})\right\rangle +\frac{L\eta_{t}^{2}}{2}\left\Vert \theta_{t}+\nabla f(x_{t})\right\Vert ^{2}\\
= & -\eta_{t}\left\Vert \nabla f(x_{t})\right\Vert ^{2}-\eta_{t}\left\langle \nabla f(x_{t}),\theta_{t}\right\rangle +\frac{L\eta_{t}^{2}}{2}\left\Vert \theta_{t}\right\Vert ^{2}+\frac{L\eta_{t}^{2}}{2}\left\Vert \nabla f(x_{t})\right\Vert ^{2}+L\eta_{t}^{2}\left\langle \nabla f(x_{t}),\theta_{t}\right\rangle \\
= & -\left(\eta_{t}-\frac{L\eta_{t}^{2}}{2}\right)\left\Vert \nabla f(x_{t})\right\Vert ^{2}+\frac{L\eta_{t}^{2}}{2}\left\Vert \theta_{t}\right\Vert ^{2}+\left(L\eta_{t}^{2}-\eta_{t}\right)\left\langle \nabla f(x_{t}),\theta_{t}\right\rangle \\
= & -\left(\eta_{t}-\frac{L\eta_{t}^{2}}{2}\right)\left\Vert \nabla f(x_{t})\right\Vert ^{2}+\frac{L\eta_{t}^{2}}{2}\left\Vert \theta_{t}\right\Vert ^{2}+\left(L\eta_{t}^{2}-\eta_{t}\right)\left\langle \nabla f(x_{t}),\theta_{t}^{u}+\theta_{t}^{b}\right\rangle .
\end{align*}
Using Cauchy-Schwarz, we have $\left\langle \nabla f(x_{t}),\theta_{t}^{b}\right\rangle \le\frac{1}{2}\left\Vert \nabla f(x_{t})\right\Vert ^{2}+\frac{1}{2}\left\Vert \theta_{t}^{b}\right\Vert ^{2}$
thus, we derive,
\begin{align*}
\Delta_{t+1}-\Delta_{t} & \le-\left(\frac{2\eta_{t}-L\eta_{t}^{2}}{2}\right)\left\Vert \nabla f(x_{t})\right\Vert ^{2}+\frac{L\eta_{t}^{2}}{2}\left\Vert \theta_{t}\right\Vert ^{2}+\left(L\eta_{t}^{2}-\eta_{t}\right)\left\langle \nabla f(x_{t}),\theta_{t}^{u}\right\rangle \\
 & \quad+\frac{\eta_{t}-L\eta_{t}^{2}}{2}\left\Vert \nabla f(x_{t})\right\Vert ^{2}+\frac{\eta_{t}-L\eta_{t}^{2}}{2}\left\Vert \theta_{t}^{b}\right\Vert ^{2}\\
 & \le-\frac{\eta_{t}}{2}\left\Vert \nabla f(x_{t})\right\Vert ^{2}+\frac{L\eta_{t}^{2}}{2}\left\Vert \theta_{t}\right\Vert ^{2}+\left(L\eta_{t}^{2}-\eta_{t}\right)\left\langle \nabla f(x_{t}),\theta_{t}^{u}\right\rangle +\frac{\eta_{t}}{2}\left\Vert \theta_{t}^{b}\right\Vert ^{2}\\
 & \le-\frac{\eta_{t}}{2}\left\Vert \nabla f(x_{t})\right\Vert ^{2}+L\eta_{t}^{2}\left\Vert \theta_{t}^{u}\right\Vert ^{2}+\left(L\eta_{t}^{2}-\eta_{t}\right)\left\langle \nabla f(x_{t}),\theta_{t}^{u}\right\rangle +\left(L\eta_{t}^{2}+\frac{\eta_{t}}{2}\right)\left\Vert \theta_{t}^{b}\right\Vert ^{2}\\
 & \le-\frac{\eta_{t}}{2}\left\Vert \nabla f(x_{t})\right\Vert ^{2}+L\eta_{t}^{2}\left\Vert \theta_{t}^{u}\right\Vert ^{2}+\left(L\eta_{t}^{2}-\eta_{t}\right)\left\langle \nabla f(x_{t}),\theta_{t}^{u}\right\rangle +\frac{3\eta_{t}}{2}\left\Vert \theta_{t}^{b}\right\Vert ^{2}
\end{align*}
Rearranging, we obtain the lemma.

\end{proof}

\begin{proof}[Proof of Lemma \ref{lem:parameter-property-nonconvex}]
First, since $\eta=\frac{\sqrt{\Delta_{1}}T^{\frac{1-p}{3p-2}}}{8\lambda\sqrt{L}\ln\frac{4T}{\delta}}$,
$p>1$ and $\lambda\ge\left(\frac{8\ln\frac{4T}{\delta}}{\sqrt{L\Delta_{1}}}\right)^{\frac{1}{p-1}}T^{\frac{1}{3p-2}}\sigma^{\frac{p}{p-1}}$
\begin{align*}
\eta\lambda^{p} & =\frac{\sqrt{\Delta_{1}}T^{\frac{1-p}{3p-2}}}{8\sqrt{L}\ln\frac{4T}{\delta}}\lambda^{p-1}\\
 & \ge\frac{\sqrt{\Delta_{1}}T^{\frac{1-p}{3p-2}}}{8\sqrt{L}\ln\frac{4T}{\delta}}\frac{8\ln\frac{4T}{\delta}}{\sqrt{L\Delta_{1}}}T^{\frac{p-1}{3p-2}}\sigma^{p}\\
 & =\frac{\sigma^{p}}{L}
\end{align*}
which gives 
\begin{align*}
\frac{1}{L}\left(\frac{\sigma}{\lambda}\right)^{p} & \le\eta.
\end{align*}

Second, from the above, and $\lambda\ge4\sqrt{L\Delta_{1}}$ we have
\begin{align*}
\eta & =\frac{\sqrt{\Delta_{1}}}{8\sqrt{L}\ln\frac{4T}{\delta}\lambda}\le\frac{1}{32L\ln\frac{4T}{\delta}}\le\frac{1}{L}.
\end{align*}

Third, we have $\lambda\ge32^{1/p}\sigma T^{\frac{1}{3p-2}}$ hence
\begin{align*}
\left(\frac{\sigma}{\lambda}\right)^{p}T^{\frac{p}{3p-2}} & \le\frac{1}{32}.
\end{align*}

Finally, we have
\begin{align*}
TL\left(\frac{\sigma}{\lambda}\right)^{p}\lambda^{2}\eta^{2} & =TL\left(\frac{\sigma}{\lambda}\right)^{p}\left(\frac{\sqrt{\Delta_{1}}T^{\frac{1-p}{3p-2}}}{8\sqrt{L}\ln\frac{4T}{\delta}}\right)^{2}\\
 & =TL\left(\frac{\sigma}{\lambda}\right)^{p}T^{\frac{2-2p}{3p-2}}\frac{\Delta_{1}}{64L\ln^{2}\frac{4T}{\delta}}\\
 & =\left(\frac{\sigma}{\lambda}\right)^{p}T^{\frac{p}{3p-2}}\frac{\Delta_{1}}{64\ln^{2}\frac{4T}{\delta}}\\
 & \leq\frac{1}{32}\frac{\Delta_{1}}{64\ln^{2}\frac{4T}{\delta}}\\
 & \leq\frac{\Delta_{1}}{2048}\tag{ since \ensuremath{T>1} and \ensuremath{\delta<1}}.
\end{align*}
\end{proof}

\section{Property of smooth functions}
\begin{lem}
\label{lem:smooth-prop}For any $L$-smooth function $f$ such that
$f^{*}\coloneqq\inf_{x\in\R^{d}}f(x)>-\infty$, we have
\begin{align*}
\norm{\nabla f(x)}^{2} & \leq2L\left[f(x)-f^{*}\right].
\end{align*}
 
\end{lem}
\begin{proof}
By smoothness and lower bound of $f$, we have for all $x,y$
\[
f^{*}\leq f(y)\leq f(x)+\left\langle \nabla f(x),y-x\right\rangle +\frac{L}{2}\norm{y-x}^{2}.
\]
Pick $y=x-\frac{1}{L}\nabla f(x)$, we have 
\begin{align*}
f^{*} & \leq f(x)-\frac{1}{L}\norm{\nabla f(x)}^{2}+\frac{1}{2L}\norm{\nabla f(x)}^{2}\\
 & =f(x)-\frac{1}{2L}\norm{\nabla f(x)}^{2}\\
\implies\norm{\nabla f(x)}^{2} & \leq2L\left[f(x)-f^{*}\right].
\end{align*}
\end{proof}

\end{document}